\RequirePackage{fix-cm}
\documentclass[twocolumn]{svjour3_arxiv}          
\smartqed  
\usepackage[pdftex]{graphicx}

\usepackage[l2tabu, orthodox]{nag}
\usepackage{color}
\usepackage[usenames,dvipsnames]{xcolor}
\definecolor{someblue}{rgb}{0,0,.8}
\definecolor{darkgreen}{rgb}{0,.5,0}
\definecolor{DarkBlue}{rgb}{0,0,0.4}
\usepackage{amsmath, amssymb,amstext,latexsym,bbm,bbold}
\usepackage{csquotes}
\usepackage{tikz}
\usetikzlibrary{calc}
\usepackage{pgfplots}
\pgfplotsset{compat=newest}
\usepgfplotslibrary{patchplots}

\usepackage[subrefformat=parens,labelformat=parens]{subfig}
\usepackage{multirow}
\usepackage{rotating}

\usepackage{changepage}

\usepackage{stmaryrd}
\SetSymbolFont{stmry}{bold}{U}{stmry}{m}{n}

\newcommand{\average}[1]{{\left\langle #1 \right\rangle}}
\newcommand{\jump}[1]{{\left\llbracket #1 \right\rrbracket}}
\newcommand{\upwind}[1]{{\left\{ #1 \right\}}^{\mathrm{up}}}

\newcommand{\IN}{\mathbb{N}}

\newcommand{\IR}{\mathbb{R}}

\renewcommand{\vec}[1]{\mathbf{#1}}

\newcommand{\abs}[1]{\left\lvert{#1}\right\rvert}

\newcommand{\spf}[2]{{\left\langle{#1},{#2}\right\rangle}}

\usepackage{bbm}

\renewcommand{\vec}[1]{\boldsymbol{#1}}

\renewcommand{\tens}[1]{\boldsymbol{\mathsf{#1}}}

\newcommand{\einschraenkung}{\,\rule[-5pt]{0.4pt}{12pt}\,{}}

\begin{document}

\title{Generating admissible space-time meshes for moving domains in $d+1$-dimensions}


\author{Elias Karabelas         \and
        Martin Neum\"uller 
}


\institute{E. Karabelas \at
              Medical University of Graz \\
              Institute of Biophysics \\
              Harrachgasse 21/IV\\
              A-8010 Graz, Austria\\ 
              Tel.: +43-316-380 7759\\
              Fax: +43-316-380 9660\\
              \email{elias.karabelas@medunigraz.at}           
           \and
           M. Neum\"uller \at
           	  Johannes Kepler University\\	
              Institute of Computational Mathematics\\
              Altenberger Strasse 69\\
              A-4040 Linz, Austria
}

\date{}

\maketitle

\begin{abstract}
In this paper we present a di\-sco\-ntinu\-ous Ga\-le\-rk\-in finite element method for the solution of the transient Stokes equations on moving domains. For the discretization we use an interior penalty Galerkin approach in space, and an upwind technique in time. The method is based on a decomposition of the space-time cylinder into finite elements. Our focus lies on three-dimensional moving geometries, thus we need to triangulate four dimensional objects. For this we will present an algorithm to generate $d+1$-dimensional simplex space-time meshes and we show under natural assumptions that the resulting space-time meshes are admissible. Further we will show how one can generate a four-dimensional object resolving the domain movement. First numerical results for the transient Stokes equations on triangulations generated with the newly developed meshing algorithm are presented.
\end{abstract}

\section{Introduction}
\label{intro}
The finite element approximation of transient partial differential equations is in most cases based on explicit or implicit time discretization schemes. In particular the simultaneous consideration of different time steps requires an appropriate interpolation to couple the solutions at different time levels. Especially for spatial domains with a moving boundary one encounters various numerical difficulties. One usually relies on an arbitrary Lagrangian-Eulerian formulation. See for example the recent article \cite{Gawlik2014} and references therein for an overview of the ongoing discussion. In this paper we consider the application of finite elements in the whole space-time cylinder $Q$. By this we mean a decomposition of $Q$ into simplical elements. Therefore one replaces the problem of time discretization by a meshing problem. Having this, one can resolve the possible movement of the domain $\Omega$ directly. Simplicial space-time meshes have advantages over tensor-product meshes, since it is easier to decompose complex space-time meshes by those elements.

Space-time finite element methods have been applied to several parabolic model problems. Least square methods for convection-diffusion problems are considered, e.g., in \cite{Bank2013,Behr2008} and for flow problems, e.g., in \cite{Masud1997,Tezduyar1992:I,Tezduyar1992:II,Tezduyar2004,Tezduyar2006,Rendall2012}. Discontinuous Galerkin finite element methods have been applied to solve transient convection-diffusion problems in \cite{Sudirham2006}, for fluid dynamics see \cite{Van2008} and problems from solid mechanics see \cite{Miller2009,Abedi2006,Dumont2012,Abedi2010}. Very recently a paper concerning the generation of 4D simplicial meshes from a sequence of 3D MRI data has been considered in \cite{Foteinos2014}. Also rather recently, the X-FEM method has been considered in the space-time setting see \cite{Lehrenfeld2015}. In most cases, the time dependent equation is discretized in the space-time domain on \emph{space-time slabs}. This allows for local mesh refinement in the space-time domain, see for example \cite{Tezduyar2004}. 

In this paper we consider, similar to \cite{Behr2008}, a decomposition of the space-time cylinder into simplicial finite elements. In particular for spatial domains $\Omega \subsetneq \mathbb R^3$ the space-time cylinder is a four-dimensional object, which has to be decomposed into finite elements. 

In \cite{Behr2008}, a method based on Delauny's algorithm is given to construct a four-dimensional triangulation out of a given decomposition of the spatial domain $\Omega$. This method relies on the extension of the given finite elements of the triangulation of $\Omega$ to four-dimensional prisms. Afterwards a random perturbation of the resulting points is made, to ensure the admissibility of the resulting four dimensional mesh. Here, we present a different approach using similar ideas. Our method does not rely on random perturbations. Furthermore we can ensure and proof that the resulting mesh is admissible and we can also include movements of the domain boundary. We want to stress out, that our approach is still limited to a special class of boundary movements which we will describe in Section \ref{sec:mov_dom}.

We will consider Stokes flow as a motivating model problem. For the approximation of the transient Stokes equations in the space-time cylinder we consider a discontinuous Galerkin finite element method. In particular, we apply an interior penalty approach in space \cite{Arnold2002,Baumann1999,Riviere2008,D.A.D.Pietro2012}, and an upwind technique in time \cite{Thomee2006,NeumuellerThesis2013}. 

This paper is organized as follows. In Section \ref{sec:DG} we describe the discontinuous Galerkin finite element method to solve the transient Stokes equations as a model problem. The core part of this paper and the main results are given in Section \ref{sec:meshing4d} where we describe our algorithm to generate a four-dimensional triangulation out of a given three-dimensional one. In Section \ref{sec:results} we present some numerical results which underline the applicability of the proposed approach. We close the paper with some conclusions and comments on further work.
\section{Space-time discontinuous Galerkin Method}\label{sec:DG}
For any $t \in (0,T)$ let $\Omega(t) \subsetneq \mathbb R^d$ with $d=1,2,3$ be a bounded Lipschitz domain with boundary $\Gamma(t) := \partial \Omega(t)$. We assume that the boundary $\Gamma(t)$ admits the following decomposition for every $t \in (0,T)$
\begin{equation}\label{eq:bd_def}
\Gamma(t) = \Gamma_D(t) \cup \Gamma_R(t).
\end{equation}
We assume that the movement of the domain $\Omega(t)$ is known for every $t \in [0,T]$. We define the space-time cylinder $Q$ as
\begin{equation}
\nonumber Q := \left\{ (\vec x, t) \in \mathbb R^{d+1} : \vec x \in \Omega(t) ~ t \in (0,T)\right\}.
\end{equation}
Further we define the space-time mantle $\Sigma$ as
\begin{equation*}
\Sigma := \left\{ (\vec x, t) \in \mathbb R^{d+1} : \vec x \in \Gamma(t) ~ t \in (0,T)\right\}.
\end{equation*}
The decomposition \eqref{eq:bd_def} induces
\begin{equation*}
\Sigma = \Sigma_D \cup \Sigma_R.
\end{equation*}
The model problem we intend to study is governed by the transient Stokes equations. It reads as find $(\vec u,p)$ such that
\begin{equation}
\begin{aligned}
\frac{\partial}{\partial t} \vec u -\nu \Delta \vec u + \nabla p &= \vec f && \text{in } Q,\\
\mathrm{div}(\vec u) &= 0 && \text{in } Q,\\
\vec u &= \vec g_D && \text{on } \Sigma_D,\\
\label{eq:robin_BC}\nabla \vec u \cdot \vec n + \alpha_R \vec u - p \vec n &= \vec g_R && \text{on } \Sigma_R,\\
\vec u &= \vec u_0 && \text{on } \Sigma_0 := \Omega(0).
\end{aligned}
\end{equation}
\begin{remark}
In the case of a non-moving domain the definition of $Q$ and $\Sigma$ simplifies to
\begin{align*}
Q &:= \Omega \times (0,T),\\
\Sigma &:= \partial \Omega \times (0,T).
\end{align*}
\end{remark}
For deriving a discrete variational formulation we need to decompose the space-time cylinder $Q$ into simplicial elements, see \cite{Neumueller2013}. Let $\mathcal T_h$ be a sequence of decompositions
\begin{align*}
	\overline{Q} = \overline{\mathcal T}_h = \bigcup_{k=1}^N \overline{\tau}_k
\end{align*} 
into finite elements of mesh size $h_k$. For $d=1$ we have triangles, for $d=2$ we use tetrahedrons and for $d=3$ pentatopes are chosen. The generation of such triangulations from a given triangulation of $\Omega(0)$ is not trivial. We will address this topic in Section \ref{sec:meshing4d}.
\begin{definition}[Admissible decomposition]
A decomposition $\mathcal T_h$ is called \emph{admissible} iff the non-empty intersection of the closure of two finite elements is either an edge (for $d=1,2,3$), a triangle (for $d=2,3$) or a tetrahedron (for $d=3$).
\end{definition}
It is worth noting that discontinuous Galerkin methods are not restricted to admissible decompositions. However one needs additional technical assumptions, see \cite{Ortner2009}.
\begin{definition}[Interior facet]
Let $\mathcal T_h$ be a decomposition of $Q$ into finite elements $\tau_k$. For two neighboring elements $\tau_k, \tau_l \in \mathcal T_h$ we call
\begin{align*}
	\Gamma_{kl} := \overline{\tau}_k \cap \overline{\tau}_l
\end{align*} 
an \emph{interior facet} iff $\Gamma_{kl}$ is a $d$-dimensional manifold. The set of all interior facets will be defined as $\mathcal I_h$.
\end{definition}
Any interior element $\Gamma_{kl}$ has an exterior normal vector $\vec n_{kl}$ with a non-unique direction. We fix the direction of the normal vector such that $\vec n_{kl}$ is the exterior normal vector of the
element $\tau_k$ when $k < l$. So the direction of the normal vector $\vec n_{kl}$ depends on
the ordering of the finite elements, but the variational formulation which we are going to
use will be independent of this ordering.
\begin{definition}
	Let $\Gamma_{kl} \in \mathcal I_h$ be an interior facet with outer normal $\vec n_k = (\vec n_{x,k}, n_{t,k})^\top \in \mathbb R^{d+1}$ for $\tau_k$ and $\vec n_l = -\vec n_k$ for $\tau_l$. For a given function $\phi$  smooth enough restricted to either $\tau_k$ or $\tau_l$ one defines :
	\begin{itemize}
		\item The \emph{jump} across $\Gamma_{kl}$ as
		\begin{align*}
		\jump{\phi}_{kl}:= \phi\einschraenkung_{\tau_k}\vec n_k + \phi\einschraenkung_{\tau_l} \vec n_l.
		\end{align*}
		\item The \emph{space jump} across $\Gamma_{kl}$ as
		\begin{align*}
		\jump{\phi}_{\vec x,kl}:= \phi\einschraenkung_{\tau_k}\vec n_{\vec x,k} + \phi\einschraenkung_{\tau_l} \vec n_{\vec x,l}.
		\end{align*}
		\item The \emph{time jump} across $\Gamma_{kl}$ as
		\begin{align*}
		\jump{\phi}_{t,kl}:= \phi\einschraenkung_{\tau_k} n_{t,k} + \phi\einschraenkung_{\tau_l} n_{t,l}.
		\end{align*}
		\item The \emph{average} of $\phi$ on $\Gamma_{kl}$ as
		\begin{align*}
		\average{\phi}_{kl} := \frac{1}{2}\left( \phi\einschraenkung_{\tau_k} + \phi\einschraenkung_{\tau_l}\right).
		\end{align*}
		\item The \emph{upwind} in time direction of $\phi$ is defined as
		\begin{align*}
		\upwind{\phi}_{kl}:= \begin{cases}
		\phi\einschraenkung_{\tau_k} & \text{if } n_{k,t} > 0 \\
		0			     & \text{if } n_{k,t} = 0 \\
		\phi\einschraenkung_{\tau_l} & \text{if } n_{k,t} < 0
		\end{cases}
		\end{align*}
	\end{itemize}
\end{definition}
Let $p,q \in \mathbb N_0$. Then one defines the spaces of piecewise polynomials
\begin{align*}
V_h^{p} &:= \left[S_h^p(\mathcal T_h)\right]^{d} =\left\{\vec v_h \in [\mathrm{L}^2(Q)]^d : \vec v_h \einschraenkung_{\tau_l} \in [\mathbb P_p(\tau_l)]^d~\right. \\ &\left. \text{for all }\tau_l \in \mathcal T_h,~\vec v_h \einschraenkung_{\Sigma_D} = \vec{0} \right\},\\
Q_h^q &:= \left\{ q_h \in \mathrm{L}^2(Q) : q_h\einschraenkung_{\tau_l} \in \mathbb P_q(\tau_l)~\text{for all }\tau_l \in \mathcal T_h \right\}.
\end{align*}
Inspired by works in \cite{NeumuellerThesis2013,D.A.D.Pietro2012} we will use the following  bilinear form defined for $\vec u_h , \vec v_h \in V_h^p(\mathcal T_h)$:
\begin{align*}
	A(\vec u_h, \vec v_h) := b_T(\vec u_h, \vec v_h)+a_h(\vec u_h, \vec v_h).
\end{align*}
The individual components read as
\begin{align*}
a_h(\vec u_h, \vec v_h) &:= \nu\sum_{l=1}^{N}\int_{\tau_l} \nabla_{\vec x} \vec u_h : \nabla_{\vec x} \vec v_h~dq\\
&-\nu\sum_{\Gamma_{kl} \in \mathcal I_h}\int_{\Gamma_{kl}}\average{\nabla_{\vec x} \vec u_h}_{\Gamma_{kl}}\jump{\vec v_h}_{\Gamma_{kl},\vec x}~ds_q\\&-\nu\sum_{\Gamma_{kl} \in \mathcal I_h}\int_{\Gamma_{kl}}\average{\nabla_{\vec x} \vec v_h}_{\Gamma_{kl}}\jump{\vec u_h}_{\Gamma_{kl},\vec x}~ds_q,\\
&+\sum_{\Gamma_{kl} \in \mathcal I_h} \frac{\sigma_u}{\overline{h}_{kl}} \int_{\Gamma_{kl}} \jump{\vec u_h}_{\Gamma_{kl},\vec x}\jump{\vec v_h}_{\Gamma_{kl},\vec x}~ds_q\\
&+\int_{\Sigma_R} \alpha_R(\vec x,t) \vec u_h \cdot \vec v_h ~ds_q,
\end{align*}
and
\begin{align*}
b_T(\vec u_h,\vec v_h)&:=\sum\limits_{l=1}^{N} -\int_{\tau_l}\vec u_h \cdot \frac{\partial}{\partial t} \vec v_h~dq + \int_{\Sigma_T}\vec u_h \cdot \vec v_h ds_q \\&+ \sum_{\Gamma_{kl}\in \mathcal I_h} \int_{\Gamma_{kl}} \upwind{\vec u_h} \jump{\vec v_h}_{\Gamma_{kl},t}~ds_q
\end{align*}
for a given velocity stabilization parameter $\sigma_u > 0$. Furthermore we define the following pressure bilinear forms for $\vec v_h \in V_h^p(\mathcal T_h)$ and  $(p_h , q_h) \in Q_h^q(\mathcal T_h) \times Q_h^q(\mathcal T_h)$ :
\begin{align*}
	b_p(\vec v_h,p_h)&:=\sum\limits_{l=1}^{N} \int_{\tau_l} p_h \mathrm{div}(\vec v_h)~dq \\&- \sum_{\Gamma_{kl}}\int_{\Gamma_{kl}}\average{p_h}_{\Gamma_{kl}} \jump{\vec v_h}_{\Gamma_{kl}, \vec x}~ds_q,\\
	d_p(p_h, q_h) &:= \sum_{\Gamma_{kl}\in\mathcal I_h}\sigma_p \overline{h}_{kl}\int_{\Gamma_{kl}}\jump{p_h}_{\Gamma_{kl},\vec x}\jump{q_h}_{\Gamma_{kl},\vec x}~ds_q
\end{align*}
for a given pressure stabilization parameter $\sigma_p$. In all the bilinear forms defined above we have used $\overline{h}_{kl} := \frac{1}{2}(h_k+h_l)$.
Hence we have to find $\vec u_h^0 \in V_h^{p}(\mathcal T_h)$ and $p_h \in Q_h^q(\mathcal T_h)$ such that
\begin{align}
\label{eq:stokes_discrete:1}A(\vec u_h^0,\vec v_h)-b_p(\vec v_h,p_h) &= \spf{\vec f}{\vec v_h}_Q +\spf{\vec u_0}{\vec v_h}\\&-A(\vec u_g^h,\vec v_h),\\
\label{eq:stokes_discrete:2}b_p(\vec u_h^0,q_h)+d_p(p_h,q_h) &= -b_p(\vec u_g^h,q_h).
\end{align}
Here we used an discrete extension $\vec u_g^h$ of the given Dirichlet data, for example a $L^2$-projection.

\section{Triangulations in $d+1$ dimensions}\label{sec:meshing4d}

In this section we will introduce an algorithm to decompose a hyperprism into simplices to generate a $d+1$ simplex space-time mesh. Moreover we will show that the resulting mesh is admissible if the nodes of the simplices from the initial mesh are sorted in a special way.

\subsection{Tensor product extensions}

A simple idea for constructing a space-time mesh for a given three-dimensional simplicial spatial mesh is to extrude the mesh in time direction by a tensor product extension, see also Figure \ref{fig:tensor_extension}. Afterwards we decompose the upcoming prisms or so called hyperprisms into simplicial elements. 

\begin{figure}[htbp]
	\centering
	   \includegraphics[width=0.45\linewidth]{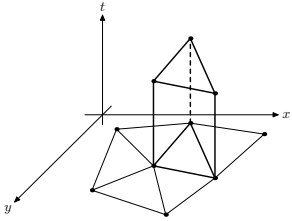}
	\caption{Tensor extension of a two-dimensional simplex}
	\label{fig:tensor_extension}
\end{figure}

Before we can start we need a precise definition of a $d$-dimensional simplex.
\begin{definition}[$d$-dimensional simplex]
  Let \begin{equation*}
  \{ \vec p_1,\ldots,\vec p_{d+1} \} \subset \IR^d,
  \end{equation*}
  $d \in \IN$ be a set of nodes, then a $d$-dimensional simplex $S_d$ is defined as
  \[ S_d := [ \vec p_1,\ldots,\vec p_{d+1} ] := \mathrm{conv}(\{ \vec p_1,\ldots,\vec p_{d+1} \}),\]
  where $\mathrm{conv}(\cdot)$ is the convex hull of a set of nodes. Note that we also fix the ordering of the nodes in the definition of a $d$-dimensional simplex.
\end{definition}

Now we can extrude one simplex in time direction and we obtain the following definition.

\begin{definition}[Hyperprism]\label{def_hyperprism}
For a given simplex $S_d = [ \vec p_1,\ldots,\vec p_{d+1} ]$ the tensor product extension in time direction for a given time interval $[0,\tau]$ or the so called hyperprism $H_{d+1}$ is given by
\begin{multline*}
P_{d+1} := [ \vec p_1,\ldots,\vec p_{d+1}; \tau ] \\:= \mathrm{conv}(\{ \vec p_1^\prime,\ldots,\vec p_{d+1}^\prime,\vec p_1^{\prime\prime},\ldots,\vec p_{d+1}^{\prime\prime}  \}) \subset \IR^{d+1},
\end{multline*}
with 
\begin{align*}
\vec p_i^\prime &:= (\vec p_i^\top,0)^\top, \\
\vec p_i^{\prime\prime} &:= (\vec p_i^\top,\tau)^\top,
\end{align*}
for $i = 1,\ldots,d+1$.
\end{definition}

\subsection{Decomposing Hyperprisms}

In this section we will give an algorithm to decompose the hyperprisms given in Definition \ref{def_hyperprism} into simplices. 
\begin{definition}[Decomposed hyperprism]
 Let $S_d$ be a given simplex and $P_{d+1}$ the hyperprism with respect to the simplex $S_d$ and $\tau > 0$. Then we define the following simplices
 \begin{equation}\label{equ_decomposed_hyperprism}
 \begin{aligned}
    S_{d+1}^1 &:= [ \vec p_1^\prime,\vec p_2^\prime,\vec p_3^\prime,\ldots,\vec p_{d+1}^\prime,\vec p_1^{\prime\prime} ],\\
    S_{d+1}^2 &:= [ \vec p_2^\prime, \vec p_3^\prime, \ldots,\vec p_{d+1}^\prime, \vec p_1^{\prime\prime}, \vec p_2^{\prime\prime} ],\\
    S_{d+1}^3 &:= [ \vec p_3^\prime, \ldots,\vec p_{d+1}^\prime, \vec p_1^{\prime\prime}, \vec p_2^{\prime\prime}, \vec p_3^{\prime\prime} ],\\
    &\;\;\;\vdots\\
    S_{d+1}^{d+1} &:= [ \vec p_{d+1}^\prime, \vec p_1^{\prime\prime}, \vec p_2^{\prime\prime}, \vec p_3^{\prime\prime}, \ldots, \vec p_{d+1}^{\prime\prime} ].
 \end{aligned}
\end{equation}
Furthermore we define the set of simplices $\mathcal{T}_P(S_d,\tau) := \{ S_{d+1}^1, \ldots, S_{d+1}^{d+1} \}$.
\end{definition}

Note, that the ordering of the nodes of a hyperpism $P_{d+1}$ is essential for the resulting decomposition (\ref{equ_decomposed_hyperprism}). In order to ensure that the simplices $S_{d+1}^1, \ldots, S_{d+1}^{d+1}$ defined in (\ref{equ_decomposed_hyperprism}) decompose the hyperpism $P_{d+1}$ we need the following lemma.
\begin{lemma}\label{lemma_decomposition_hyperprism}
  Let $P_{d+1}$ be some given hyperprism with respect to the simplex $S_d$ and $\tau > 0$. Then the set of simplices 
  \begin{equation*}
	\mathcal{T}_P(S_d,\tau) = \{ S_{d+1}^1, \ldots, S_{d+1}^{d+1} \}
  \end{equation*}
  defined in (\ref{equ_decomposed_hyperprism}) is an admissible decomposition of the hyperprism $P_{d+1}$.
\end{lemma}
\begin{proof}
  By construction the set of simplices $\mathcal{T}_P(S_d,\tau) = \{ S_{d+1}^1, \ldots, S_{d+1}^{d+1} \}$ is admissible. Furthermore, every simplex $S_{d+1}^i$ for $i=1,\ldots,d+1$ is contained in the hyperprism $P_{d+1}$ since $P_{d+1}$ is convex. It remains to show, that the union of all simplices $\mathcal{T}_P(S_d,\tau)$ is equal to the hyperprism, i.e. we have to show, that the volume of the union of all simplices $\mathcal{T}_P(S_d,\tau)$ coincides with the volume of the hyperprism. To do so, we transform the hyperprism $P_{d+1}$ to a reference hyperprism $\hat{P}_{d+1}$ where we easily can compute all the volume terms. For this, we define the reference Simplex $\hat{S}_d \subset \IR^d$ as
  \[ \hat{S}_d := [ \vec e_0, \vec e_1, \ldots, \vec e_d  ] = \mathrm{conv}(\{ \vec e_0, \vec e_1, \ldots, \vec e_d \}), \]
  with
  \begin{align*}
    \vec e_0 &:= (0,0,\ldots,0,0)^\top,\\
    \vec e_1 &:= (1,0,\ldots,0,0)^\top,\\
    \vec e_2 &:= (0,1,\ldots,0,0)^\top,\\
    &\;\;\;\vdots\\
    \vec e_d &:= (0,0,\ldots,0,1)^\top.
  \end{align*}
  Then we define the reference hyperprism $\hat{P}_{d+1}$ as
  \[ \hat{P}_{d+1} := [ \vec e_0,\ldots,\vec e_{d+1}; 1 ]. \]
  With the standard affine transformation we have a bijective mapping between the reference hyperprism $\hat{P}_{d+1}$ and the hyperprism $P_{d+1}$. This affine transformation consists of the standard transformation for $d$-dimensional simplices and a scaling in time direction. So we only have to compare the volume for the reference hyperprism. Now the volume of the reference simplex $\hat{S}_d$ is given by $\abs{\hat{S}_d} = \frac{1}{d!}.$ Hence the volume of the reference hyperprism is
  \[ \abs{\hat{P}_{d+1}} = \frac{1}{d!}. \]
  The simplices of our decomposition in the reference domain are given by
   \begin{equation*}
 \begin{aligned}
    \hat{S}_{d+1}^1 &:= [ \vec e_0^\prime,\vec e_1^\prime,\vec e_2^\prime,\ldots,\vec e_{d}^\prime, \vec e_0^{\prime\prime} ],\\
    \hat{S}_{d+1}^2 &:= [ \vec e_1^\prime, \vec e_2^\prime, \ldots, \vec e_{d}^\prime, \vec e_0^{\prime\prime}, \vec e_1^{\prime\prime} ],\\
    \hat{S}_{d+1}^3 &:= [ \vec e_2^\prime, \ldots,\vec e_{d}^\prime, \vec e_0^{\prime\prime}, \vec e_1^{\prime\prime}, \vec e_2^{\prime\prime} ],\\
    &\;\;\;\vdots\\
    \hat{S}_{d+1}^{d+1} &:= [ \vec e_{d}^\prime, \vec e_0^{\prime\prime}, \vec e_1^{\prime\prime}, \vec e_2^{\prime\prime}, \ldots, \vec e_{d}^{\prime\prime} ].
 \end{aligned}
\end{equation*}
It is easy to see, that these simplices have the same volume, i.e.
\[ \abs{\hat{S}_{d+1}^i} = \frac{1}{(d+1)!}, \quad \text{for } i=1,\ldots,d+1. \]
Hence we have
\[ \abs{\bigcup_{i=1}^{d+1} \hat{S}_{d+1}^i} = (d+1) \frac{1}{(d+1)!} = \frac{1}{d!} = \abs{\hat{P}_{d+1}}, \]
which completes the proof.
\end{proof}

\subsection{Admissible tensor product triangulations}

For a given $d$-dimensional triangulation $\mathcal{T}_h$ we now want to construct a tensor product extension by applying the algorithm (\ref{equ_decomposed_hyperprism}) for every simplex of the simplicial mesh $\mathcal{T}_h$. With Lemma \ref{lemma_decomposition_hyperprism} we know, that every hyperprism can be decomposed admissible into simplex elements. In this section we want to formulate conditions such that the overall space-time mesh is admissible. For this we need that the nodes of the simplices are ordered in a special way.
\begin{definition}[Consistently numbered]
 Let 
 \begin{equation*}
	\mathcal{T}_h = \{ S_d^i : S_d^i = [\vec p_1^i,\ldots,\vec p_{d+1}^i] \},
 \end{equation*}
 be an admissible $d$-dimensional simplex mesh. Then $\mathcal{T}_h$ is called \emph{consistently numbered}, iff for any two simplices $S_d^i, S_d^j \in \mathcal{T}_h$ with non-empty intersection, i.e. $S_d^i \cap S_d^j \neq \emptyset$, there exists indicies $k_1<\ldots<k_n$ and $\ell_1<\ldots<\ell_n$ with $n \in \IN, n \leq d+1$, such that
 \[ S_d^i \cap S_d^j = [\vec p_{k_1}^i,\ldots,\vec p_{k_n}^i] \equiv [\vec p_{\ell_1}^j,\ldots,\vec p_{\ell_n}^j]. \]
 Here ``$=$'' means that the two sets are the same and ``$\equiv$'' means that the two sets are equal and that also the numbering of the nodes is the same, i.e. $\vec p_{k_1}^i=\vec p_{\ell_1}^j,\ldots,\vec p_{k_n}^i=\vec p_{\ell_n}^j$.
\end{definition}
The definition of a consistently numbered triangulation can also be found in \cite{Bey2000} and it is important for the refinement of $d$-dimensional simplices, especially for $d \geq 4$. If an admissible mesh is consistently numbered we can prove the next Theorem.
\begin{theorem}
 Let $\mathcal{T}_h$ be an admissible $d$-dimensional triangulation which is consistently numbered and let $\tau>0$. Furthermore let 
 \[ \mathcal{T}_{h,\tau} := \{\mathcal{T}_P(S_d,\tau): S_d \in \mathcal{T}_h\} \]
 be the $(d+1)$-dimensional simplex mesh resulting by decomposing every hyperprism with the algorithm given in (\ref{equ_decomposed_hyperprism}). Then the space-time mesh $\mathcal{T}_{h,\tau}$ is admissible.
\end{theorem}
\begin{proof}
 With Lemma \ref{lemma_decomposition_hyperprism} we know, that every hyperprism is decomposed admissible into simplices. To obtain a global admissible mesh we have to prove, that the tensor product triangulations $\mathcal{T}_P(S_d^i,\tau)$ and $\mathcal{T}_P(S_d^j,\tau)$ for each neighboring elements $S_d^i, S_d^j \in \mathcal{T}_h$ are matching. Let
 \begin{equation*}
	S_d^i = [\vec p_1^i,\ldots,\vec p_{d+1}^i] \quad\text{and}\quad S_d^j = [\vec p_1^j,\ldots,\vec p_{d+1}^j]
 \end{equation*}
 with $S_d^i, S_d^j \in \mathcal{T}_h$ be some neighboring simplices and
 \begin{equation*}
 	P_d := P_{d+1}^i \cap P_{d+1}^j,
 \end{equation*}
 with
 \begin{equation*}
 P_{d+1}^i := [\vec p_1^i,\ldots,\vec p_{d+1}^i;\tau]~\text{and}~P_{d+1}^j := [\vec p_1^j,\ldots,\vec p_{d+1}^j;\tau]
 \end{equation*}
 be the intersecting hyperprism and
 \begin{align*}
\mathcal{T}_P^i &:= \{ S_{d+1} \cap P_d : S_{d+1} \in \mathcal{T}_P(S_d^i,\tau) \},\\
\mathcal{T}_P^j &:= \{ S_{d+1} \cap P_d : S_{d+1} \in \mathcal{T}_P(S_d^j,\tau) \},
 \end{align*} 
be the corresponding triangulations of $P_d$ obtained by $\mathcal{T}_P(S_d^i,\tau)$ and $\mathcal{T}_P(S_d^j,\tau)$. It remains to show, that the intersecting hyperprism $P_d$ is decomposed in the same way from both sides, i.e. that $\mathcal{T}_P^i = \mathcal{T}_P^j$. Since $\mathcal{T}_h$ is consistently numbered, there exists indices 
 $k_1<\ldots<k_n$ and $\ell_1<\ldots<\ell_n$ with $n = d$, such that
 \begin{multline}\label{theorem_admissible_equ_proof}
  S_d^i \cap S_d^j = S_{d-1}^i := [\vec p_{k_1}^i,\ldots,\vec p_{k_n}^i] \\ \equiv S_{d-1}^j := [\vec p_{\ell_1}^j,\ldots,\vec p_{\ell_n}^j].
 \end{multline}
 Therefore, the intersecting simplex $S_d^i \cap S_d^j$ is obtained by simply removing the nodes from $S_d^i$ or $S_d^j$ which are not shared together and furthermore they have the same ordering of the nodes. For the intersecting hyperpism $P_d$ the decompositions from both sides $\mathcal{T}_P^i$ and $\mathcal{T}_P^j$ are given by removing the nodes which are not shared together from the formula (\ref{equ_decomposed_hyperprism}) and with (\ref{theorem_admissible_equ_proof}) we have
 \[ \mathcal{T}_P^i = \mathcal{T}_P(S_{d-1}^i, \tau) \quad \text{and} \quad \mathcal{T}_P^j = \mathcal{T}_P(S_{d-1}^j, \tau). \]
 Since in equation (\ref{theorem_admissible_equ_proof}) also the node ordering of $S_{d-1}^i$ and $S_{d-1}^j$ is the same we also obtain
 \[ \mathcal{T}_P(S_{d-1}^i, \tau) = \mathcal{T}_P(S_{d-1}^j, \tau), \]
 which implies that $\mathcal{T}_P^i = \mathcal{T}_P^j$.
\end{proof}

\begin{remark}
To obtain an admissible space-time mesh $\mathcal{T}_{h,\tau}$ we only have to ensure, that the nodes of the spatial mesh $\mathcal{T}_h$ are consistently numbered. This can be easily obtained by sorting for each simplex $S_d \in \mathcal{T}_h$ the local nodes with respect to the global node numbers.
\end{remark}

\subsection{Tensor product triangulations for moving domains}\label{sec:mov_dom}

If the movement of a computational domain is known in advance we can generate admissible space-time meshes by applying the methods from above. The idea is to move the points at the top of the tensor-product extension. Assuming that the displacement of points on the boundary $\Gamma(t)$ is governed by a function
\begin{align*}
\vec g_\text{mov}(\vec X,t)&\colon \Gamma(0) \times (0,T) \to \mathbb R^d.
\end{align*}
Then a point $\vec x \in \Gamma(t)$ can be written as
\begin{align*}
\vec x = \vec X + \vec g_\text{mov}(\vec X,t),
\end{align*}
where $\vec X \in \Gamma(0)$. Recall the definition of a hyperprism in Definition \ref{def_hyperprism}. Instead of using $\vec p_i^{\prime\prime} := (\vec p_i^\top , \tau)^\top$ on the surface we can apply the displacement and use $\vec p_i^{\prime\prime}:=(\vec p_i^\top + \vec g_\text{mov}(\vec p_i,\tau)^\top,\tau)^\top$ for all boundary points of the simplex mesh that are subject to a movement. The remaining generation of the 4D mesh stays untouched. For boundary movements that are of small magnitude and do not change the topology of the initial geometry this can be sufficient. For stronger yet topology preserving movements this concept would create degenerating simplex elements. A remedy to this is to use the movement $\vec g_\text{mov}$ as Dirichlet datum for a vector Laplacian or a linear elasticity problem. Then the resulting displacement is applied to all simplex points in the domain.  
For more on mesh smoothing we refer to \cite{Lopez2008,Liakopoulos2006}.

In the case of stronger displacements or even topology changes re-meshing would be required and we need further meshing algorithms to connect different spatial domains in space and time, especially for four-dimensional space-time meshes this remains a future research topic. 


If the the movement of the computational domain is not known in advance we can solve the problem on a coarse spatial grid with coarse time steps to obtain a coarse approximation for the movement. Afterwards we can construct the coarse space-time mesh with the methods given in this work. By using adaptive schemes in space and time we further can refine the space-time domain adaptively and move the points in the space time domain by the computed finer approximations. Note that the movement of the points has to be only done in the range of the approximation error, which is usually small. Of course this is also considered as a further research topic.

\subsection{Visualization}\label{sec:visualization}
Here we want to address the issue of visualizing results for four-\-dimens\-ional triangulations $\mathcal T_h$. In applications it is desired to visualize results at given time instances $t_k \in [0,T]$. The main idea is to cut the decomposition $\mathcal T_h$ into a finite number of three-dimensional manifolds. For this we need to have a hyperplane to calculate the intersections with the decomposition.
\begin{definition}[Hyperplane]
	Let $\vec p_0 \in \mathbb R^4$ be arbitrary and let $\vec p_1$, $\vec p_2$, $\vec p_3$ and $\vec p_4 \in \mathbb R^4$ be linear independent. Then the set
	\begin{multline*}
		H_4 := \left\{ \vec x \in \mathbb R^4 : \vec x = \vec p_0 + \mu_1 \vec p_1 + \mu_2 \vec p_2 + \mu_3 \vec p_3 \right. \\ \left. \text{for }\mu_1,\mu_2,\mu_3 \in \mathbb R\right\}
	\end{multline*}
	is called a \emph{hyperplane}.
\end{definition}
To cut a given decomposition $\mathcal T_h$ with a hyperplane $H_4$, we have to cut every element $\tau_k \in \mathcal T_h$ with the hyperplane. For this we have to calculate for every edge $e_i = (\vec x_{i_1} , \vec x_{i_2})$, $i = 1,\ldots,10$ of $\tau_k$, the intersection with the hyperplane. A point $\vec x \in e_i$ can be written as
\begin{align*}
	\vec x = \vec x_{i_1}+\lambda\left(\vec x_{i_2}-\vec x_{i_1}\right)
\end{align*}
for a given $\lambda \in [0,1]$. Hence, an intersection point $\vec \xi_i$ of the edge $e_i$ with the hyperplane $H_4$ has to satisfy
\begin{align*}
	\vec x_{i_1} + \lambda\left(\vec x_{i_2}- \vec x_{i_1}\right) = \vec p_0 + \mu_1 \vec p_1+\mu_2 \vec p_2 + \mu_3 \vec p_3
\end{align*}
or in matrix notation
\begin{align*}
	\tens A_i := \begin{pmatrix}
	\vec p_1 & \vec p_2 & \vec p_3 & \vec x_{i_1} - \vec x_{i_2}
	\end{pmatrix}
	\begin{pmatrix}
	\mu_1 \\
	\mu_2 \\
	\mu_3 \\
	\lambda
	\end{pmatrix}
	=
	\vec x_{i_1} - \vec p_0.
	\end{align*}
The matrix $\tens A_i$ is invertible iff the vector $\vec x_{i_1} - \vec x_{i_2}$ is linear independent to the vectors $\vec p_1$, $\vec p_2$, $\vec p_3$. In fact, the matrix $\tens A_i$ is not invertible if the edge $e_i$ is parallel to the hyperplane $H_4$. In this case there exists either no intersection point or infinitely many. If the matrix is invertible we can calculate the coefficients $\mu_1$,$\mu_2$,$\mu_3$ and $\lambda \in \mathbb R$ uniquely. Let $D_k$ denote the set of all intersection points of the element $\tau_k \in \mathcal T_h$ with the hyperplane $H_4$. We distinguish two relevant cases
\begin{enumerate}
 \item If $\abs{D_k} = 4$, then the intersection points form a tetrahedron
 \item If $\abs{D_k} = 6$, then the intersection points for a general irregular prism.
\end{enumerate}
If we use the special vectors 
\begin{align*}
	\vec p_0 := t_* \vec e_t,~\vec p_1 := \vec e_x, \vec p_2 := \vec e_y, \vec p_3 := \vec e_z
\end{align*}
for a given $t_* \in [0,T]$ we can now calculate a three-dimensional object which can be visualized with existing software tools for example \cite{Paraview2015}.

\section{Numerical Results}\label{sec:results}
In this section we will present first numerical examples. Starting point is the discrete variational formulation \eqref{eq:stokes_discrete:1}-\eqref{eq:stokes_discrete:2}. This can be equivalently written as the following block system
\begin{align}
	\begin{pmatrix}
	\tens K_h & -\tens B_h^\top \\
	\tens B_h & \tens D_h
	\end{pmatrix}
	\begin{pmatrix}
	\vec U \\
	\vec P
	\end{pmatrix}
	= \begin{pmatrix}
	\vec F_1 \\
	\vec F_2
	\end{pmatrix}.
\end{align}
It is worth noting, that due to the discretization of the time derivative we have that $\tens K_h \neq \tens K_h^\top$. The four dimensional computational geometries as well as the resulting linear systems in the subsequent numerical examples were solved with the software package \textsc{Neshmet} developed by the authors. In particular we used a preconditioned GMRes method. As preconditioner we used the following:
\begin{align}
\tens P := \begin{pmatrix}
\overline{\tens K}_h & \\
 & \overline{\tens S}_h
\end{pmatrix}
\end{align}
where $\overline{\tens K}_h$ is chosen as a component-wise algebraic multigrid and $\overline{\tens S}_h$ is chosen as ILU(2)-factorization of $\tens D_h + \tens B_h \mathrm{diag}(\tens K_h)^{-1} \tens B_h^\top$. These preconditioners were taken from the HYPRE library \cite{Falgout2006}.
\subsection{Robin Boundary Conditions for Simulating Valves} \label{sec:robinBC}
In the subsequent examples we want to simulate opening and closing valves. This means, that we need to switch between and inflow and a no-slip condition. To this end we used the following configuration for \eqref{eq:robin_BC}: We set $\vec g_R \equiv \vec 0$. Further we use the following Robin coefficients for outflow
\begin{align*}
\alpha_R(x,t) := \begin{cases}
10^6 & \mbox{if } t \in [0,\frac{1}{2}) \\
0 & \mbox{if } t \in [\frac{1}{2},1]
\end{cases},
\end{align*}
and the following for inflow:
\begin{align*}
\alpha_R(x,t) := \begin{cases}
0 & \mbox{if } t \in [0,\frac{1}{2}) \\
10^{6} & \mbox{if } t \in [\frac{1}{2},1]
\end{cases}.
\end{align*}
\subsection{First Example}
In the first example we consider Stokes flow in a diaphragm pump. The geometry consists of the intersection of two cylinders. The first one has its main axis aligned with the $z$-axis with a radius of $0.8$ and ranges between $z=-0.4$ and $z=0.4$. The second cylinder has its main axis aligned with the $x$-axis with a radius of $0.2$ and ranges from $x=-1$ to $x=1$. A front view of the geometry is depicted in Figure \ref{fig:pump_front}. 
\begin{figure}[htbp]
	\centering
	\begin{tikzpicture}[scale=0.8, every node/.style={scale=0.8}]
	
	\coordinate (A) at (-1,-0.2);
	\coordinate (B) at (-0.8,-0.2);
	\coordinate (C) at (-0.8,-0.4);
	\coordinate (D) at (0.8,-0.4);
	
	\coordinate (E) at (0.8,-0.2);
	\coordinate (F) at (1.0,-0.2);
	\coordinate (G) at (1.0,0.2);
	\coordinate (H) at (0.8,0.2);
	
	\coordinate (I) at (0.8,0.35);
	\coordinate (J) at (0.75,0.4);
	\coordinate (K) at (-0.75,0.4);
	\coordinate (L) at (-0.8,0.35);
	
	\coordinate (M) at (-0.8,0.2);
	\coordinate (N) at (-1.0,0.2);
	
	\fill[black!10!white] ($4*(A)$) -- ($4*(B)$) -- ($4*(C)$) --($4*(D)$) --($4*(E)$) -- ($4*(F)$) -- ($4*(G)$) -- ($4*(H)$) -- ($4*(I)$) --($4*(J)$) -- ($4*(K)$) -- ($4*(L)$) -- ($4*(M)$) -- ($4*(N)$) --cycle;
	
	\draw[-,color=black!30!red,thick] ($4*(A)$) -- ($4*(B)$);
	\draw[-,color=black!30!red,thick] ($4*(B)$) -- ($4*(C)$);
	\draw[-,color=black!30!red,thick] ($4*(C)$) -- ($4*(D)$);
	\draw[-,color=black!30!red,thick] ($4*(D)$) -- ($4*(E)$);
	\draw[-,color=black!30!red,thick] ($4*(E)$) -- ($4*(F)$);
	\draw[-,thick] ($4*(F)$) -- ($4*(G)$);
	\draw[-,color=black!30!red,thick] ($4*(G)$) -- ($4*(H)$);
	\draw[-,color=black!30!red,thick] ($4*(H)$) -- ($4*(I)$);
	\draw[-,color=black!30!red,thick] ($4*(I)$) -- ($4*(J)$);
	\draw[-,color=black!30!green,thick] ($4*(J)$) -- ($4*(K)$);
	\draw[-,color=black!30!red,thick] ($4*(K)$) -- ($4*(L)$);
	\draw[-,color=black!30!red,thick] ($4*(L)$) -- ($4*(M)$);
	\draw[-,color=black!30!red,thick] ($4*(M)$) -- ($4*(N)$);
	\draw[-,thick] ($4*(N)$) -- ($4*(A)$);
	
	\node at (0,0) {$\Omega(0)$};
	
	\node[above] at ($2*(J)+2*(K)$) {$\Gamma_{D,m}(0)$};
	\node[right] at ($2*(F)+2*(G)$) {$\Gamma_{R,\mathrm{in}}$};
	\node[left] at ($2*(N)+2*(A)$) {$\Gamma_{R,\mathrm{out}}$};
	\node[below] at ($2*(C)+2*(D)$) {$\Gamma_{D}$};
	\end{tikzpicture}
	\caption{Front view of the initial geometry $\Omega(0)$. Blue boundaries belong to $\Gamma_D$.}
	\label{fig:pump_front}
\end{figure}
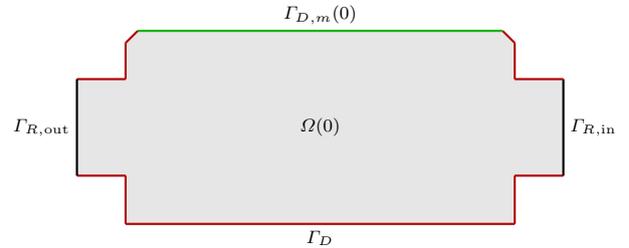
The movement of the boundary $\Gamma(t)$ was prescribed as follows
\begin{multline*}
\vec g_\text{mov} (t, \vec X) :=
\left(0.4+\mathrm{sin}^2(\pi t)\left(1-\frac{X(0)^2+X(1)^2}{0.75^2}\right)\right)\vec e_z\\ - \vec X
\end{multline*}
for $\vec X \in \Gamma_{D,m}(0)$ and $\vec 0$ else. The following boundary conditions are used:
\begin{itemize}
	\item $\vec u = \vec 0$ on $\Gamma_D$
	\item $\vec u =  \frac{\partial}{\partial t} \vec g_\text{mov}$ on $\Gamma_{D,m}(t)$
	\item On $\Gamma_{R,\text{in}}$ and $\Gamma_{R,\text{out}}$ we used the Robin boundary conditions discussed in Section \ref{sec:robinBC}
	\item The initial condition for $\vec u$ was set to $\vec u(0,\vec x) = \vec 0$.
\end{itemize}
The triangulation of the resulting 4D geometry was accomplished using the tools described in Section \ref{sec:meshing4d}. The resulting mesh consisted of 951360 pentatopes. 
\begin{figure}[htbp]
	\centering
	\begin{adjustwidth}{4em}{4em}
		\subfloat[][]{\includegraphics[width=0.49\linewidth]{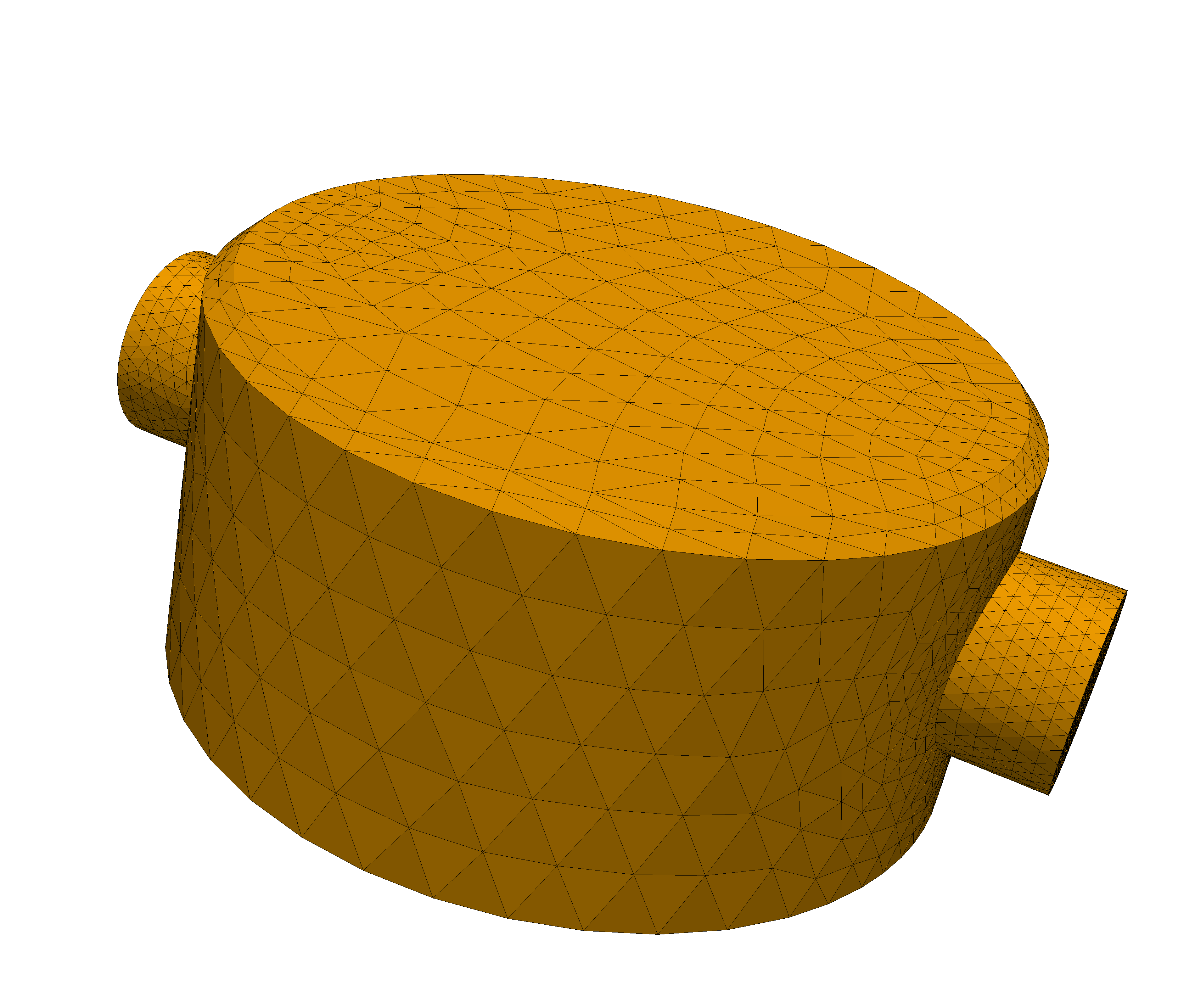}}%
		\subfloat[][]{\includegraphics[width=0.49\linewidth]{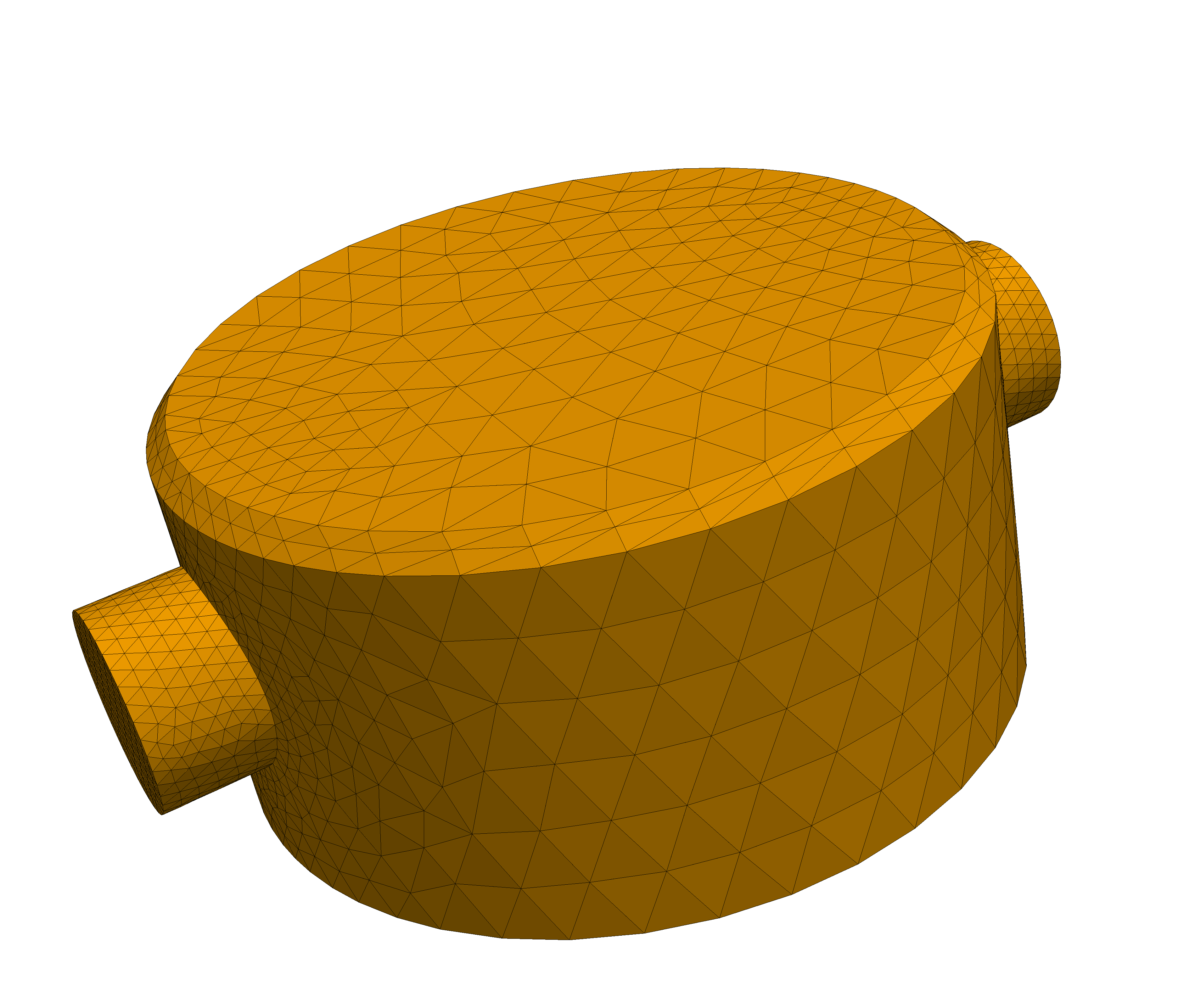}}\\%
		\subfloat[][]{\includegraphics[width=0.49\linewidth]{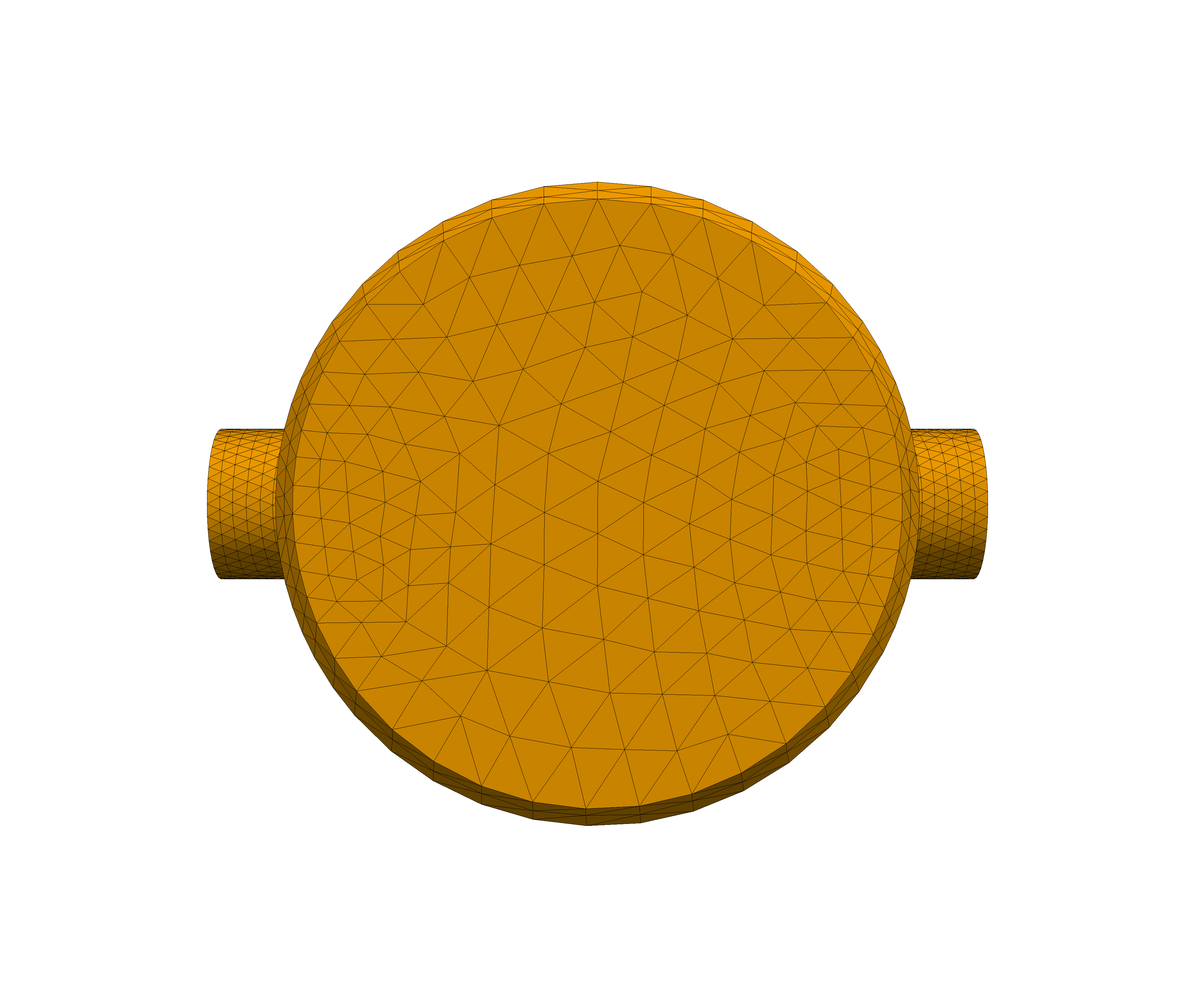}}
		\subfloat[][]{\includegraphics[width=0.49\linewidth]{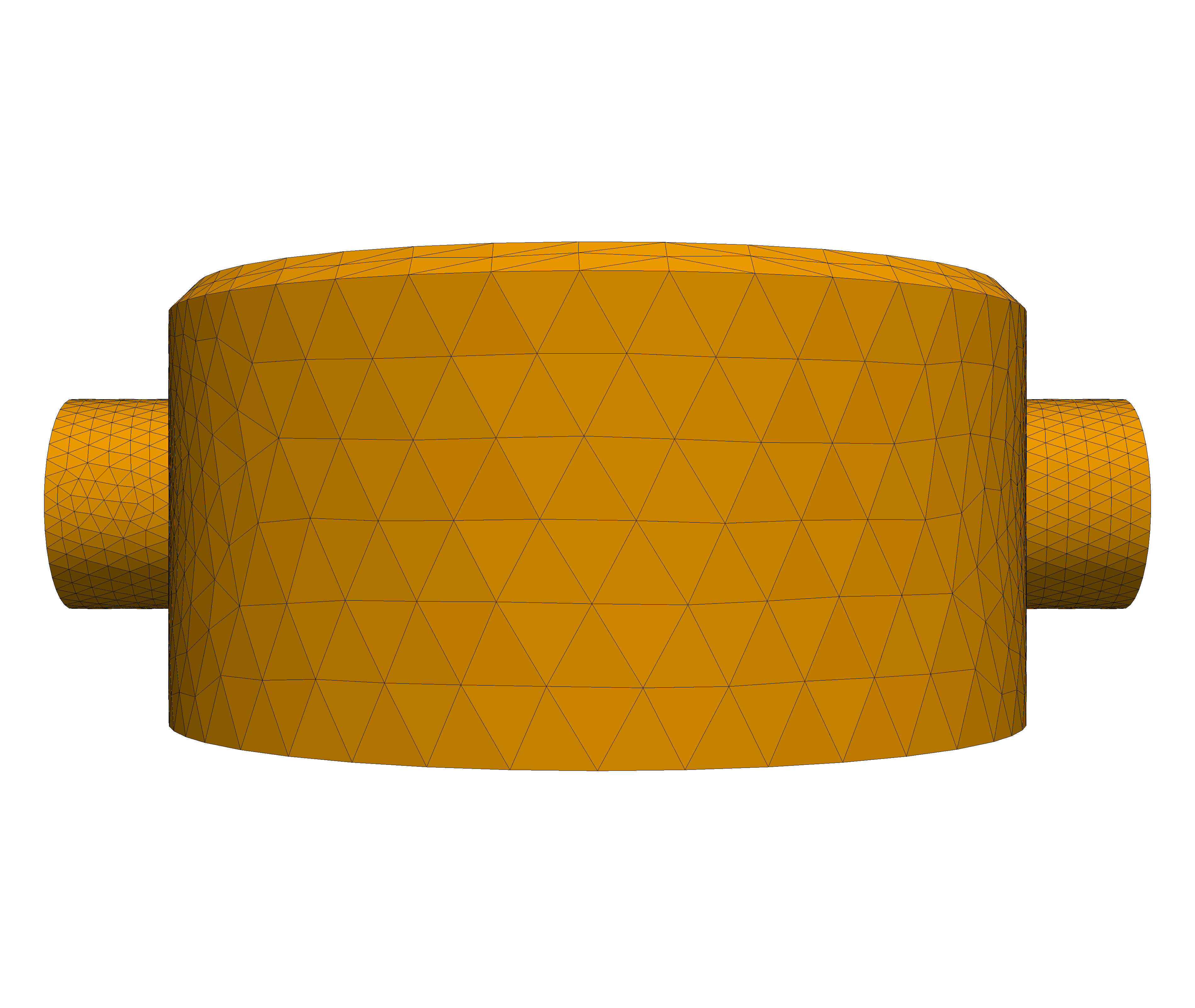}}
	\end{adjustwidth}
	\caption{Initial Triangulation $\Omega(0)$}
	\label{fig:initial_triangulation}
\end{figure}

Some snapshots of the triangulation of the moving domain are depicted in Figure \ref{fig:triangulation_snapshots}. This snapshots were generated by slicing through the 4D mesh along the time axis as described in Section \ref{sec:visualization}. 
\begin{figure}[htbp]
	\centering
	\begin{adjustwidth}{4em}{4em}
		\subfloat[][$\Omega(0.5)$]{\includegraphics[width=0.49\linewidth]{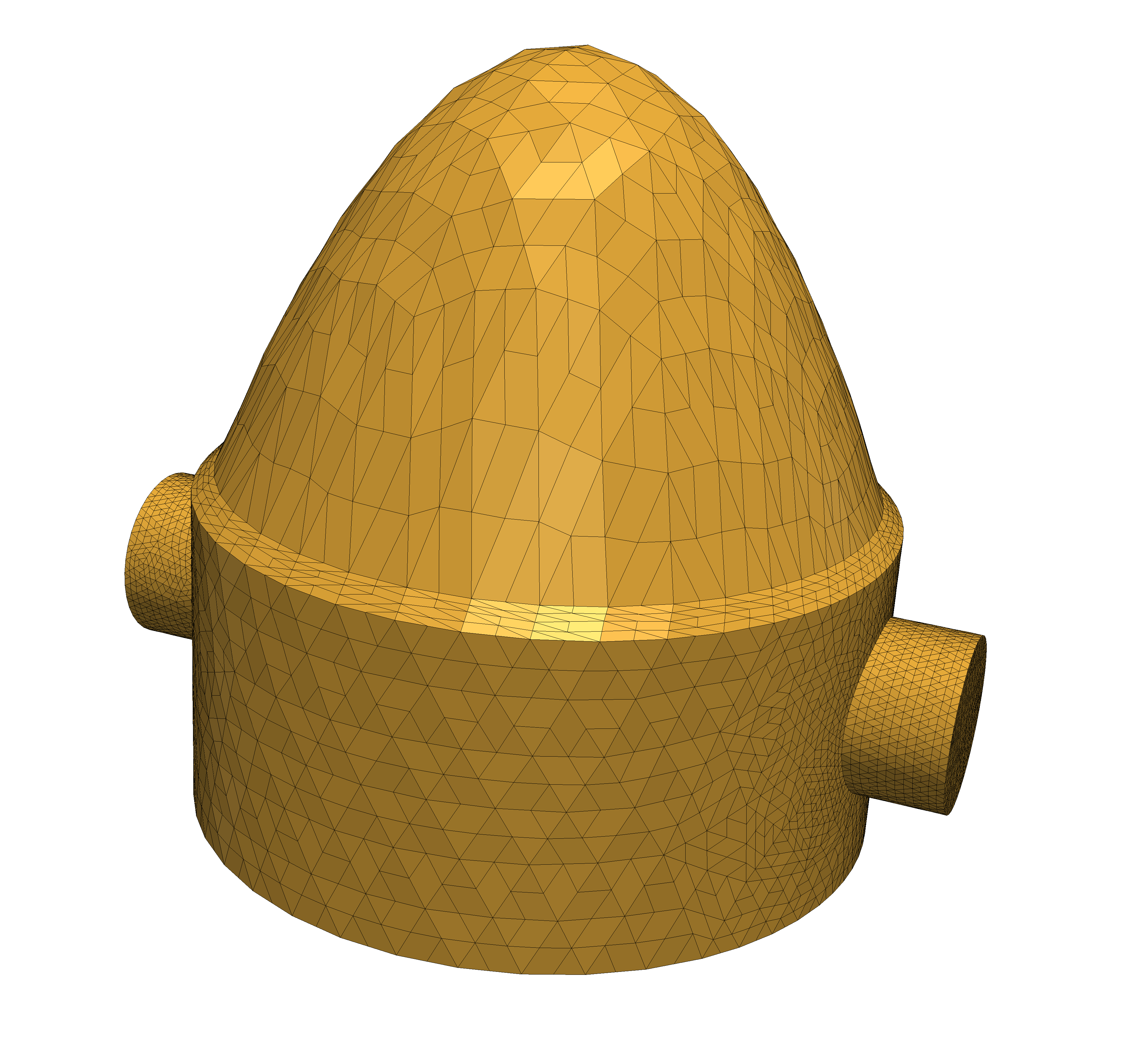}}%
		\subfloat[][$\Omega(0.5)$]{\includegraphics[width=0.49\linewidth]{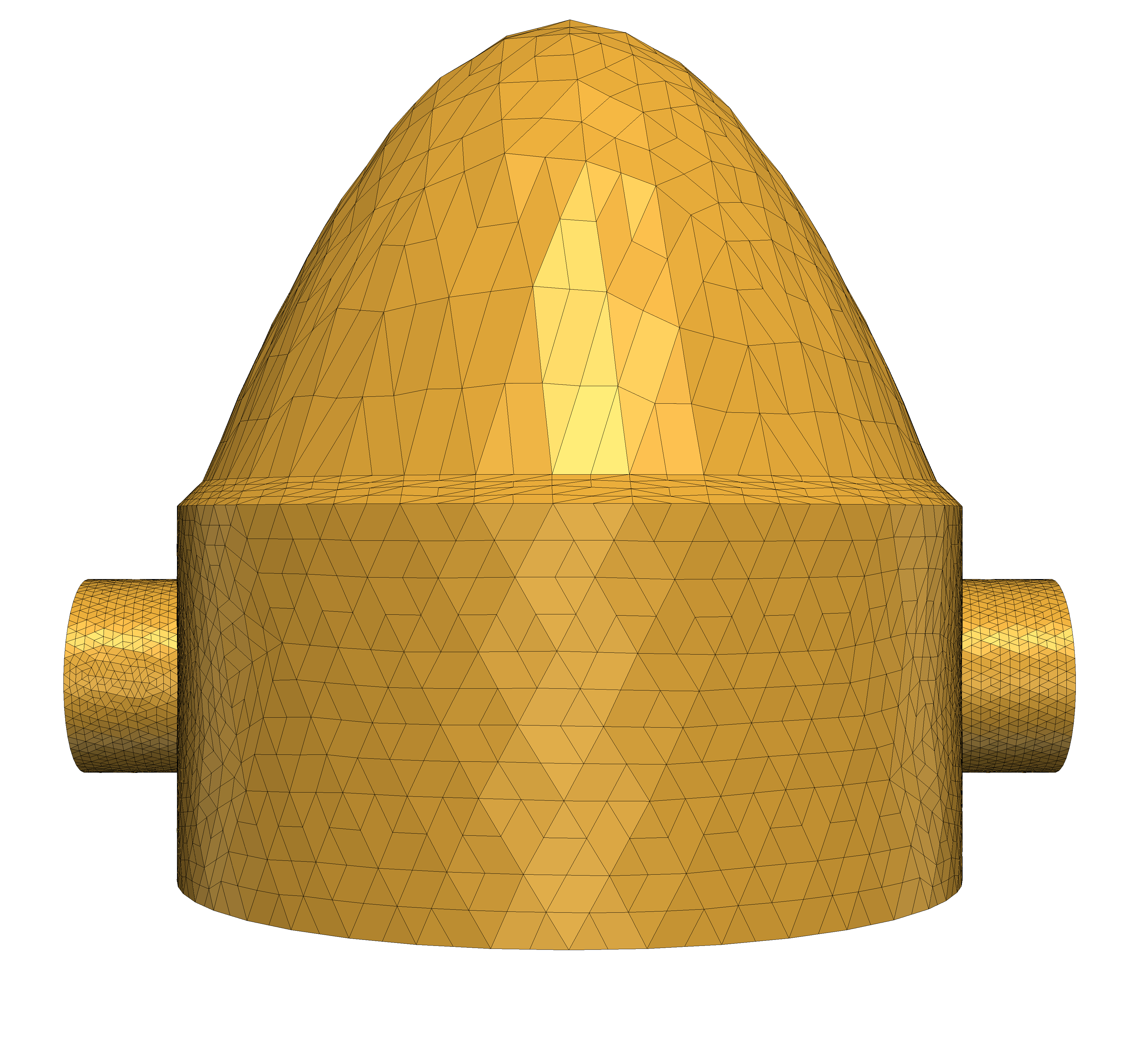}}\\%
		\subfloat[][$\Omega(0.65)$]{\includegraphics[width=0.49\linewidth]{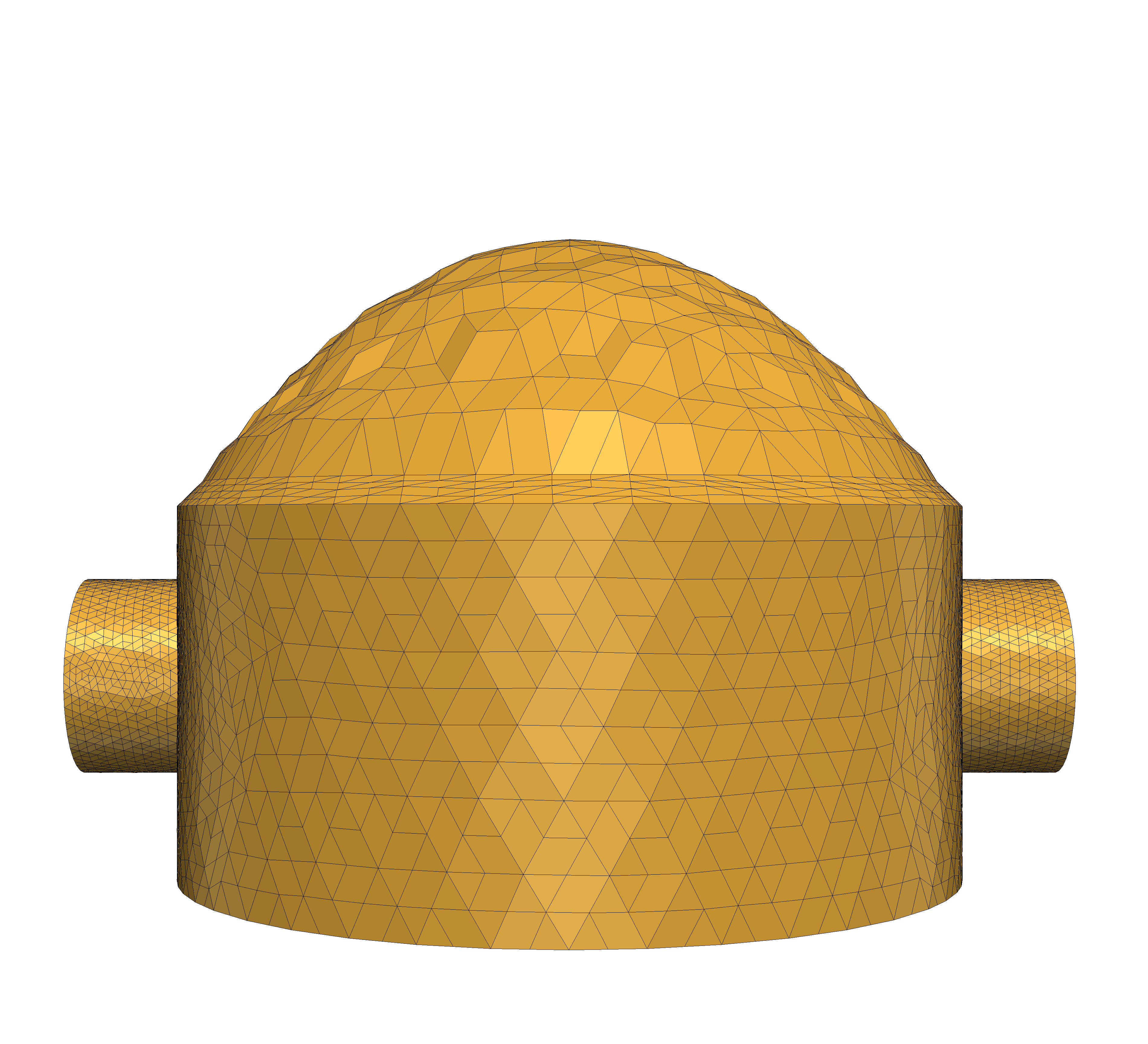}}
		\subfloat[][$\Omega(0.8)$]{\includegraphics[width=0.49\linewidth]{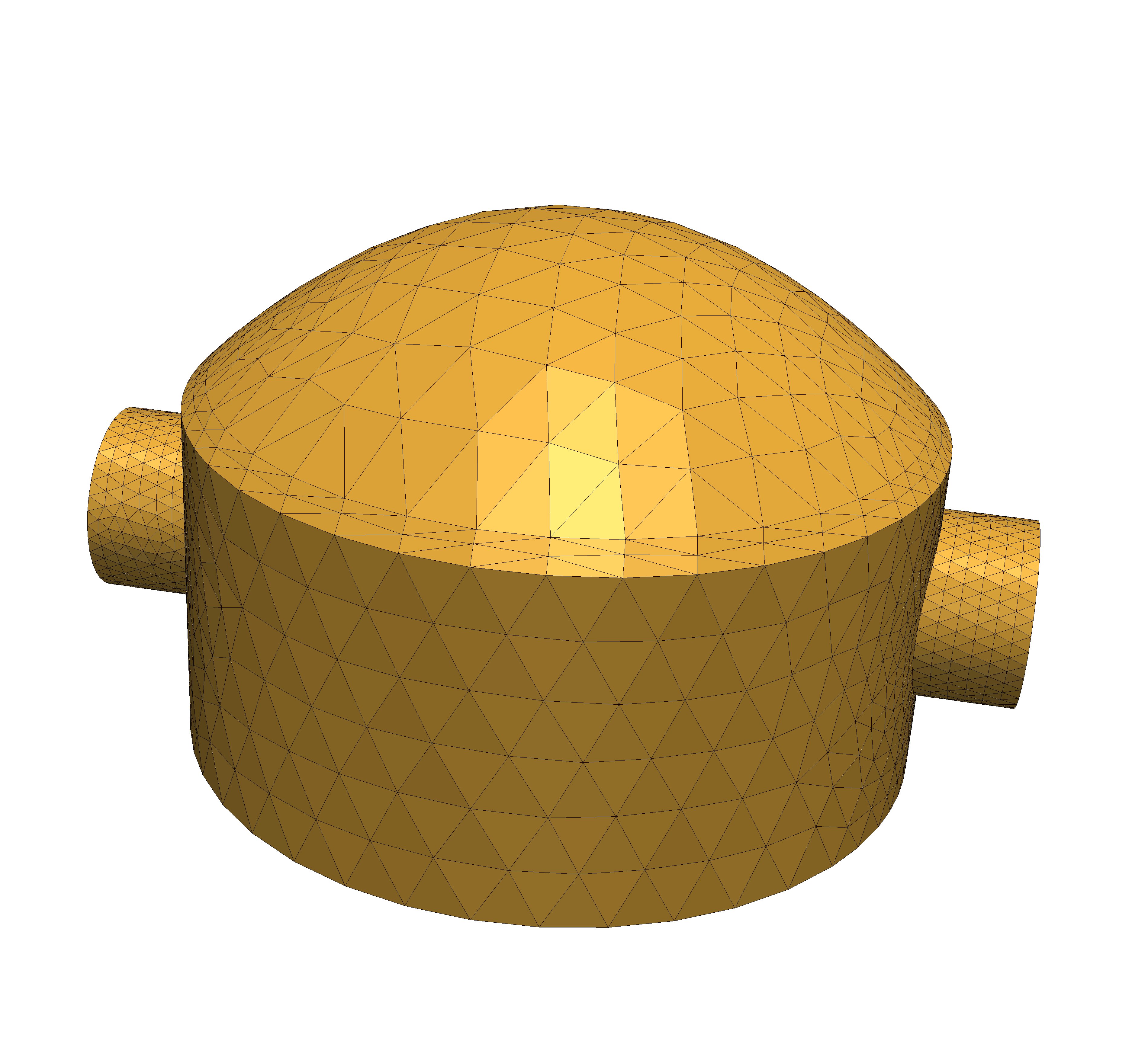}}
	\end{adjustwidth}
	\caption{Snapshots of the Triangulations $\Omega(t)$}
	\label{fig:triangulation_snapshots}
\end{figure}
The polynomial degree for $\vec u_h$ was set to $p=1$ and $q=0$ for the pressure variable $p_h$. This resulted in 14270400 degrees of freedom for $\vec u_h$ and 951360 degrees of freedom for $p_h$. We needed 95 GMRes-iterations for achieving a relative error of $1E-5$. In Figure \ref{fig:results:1} one can see the resulting flow and pressure at given time stamps, which were again produced by slicing the 4D geometry along the time axis.

\begin{figure*}[htbp]
	\centering
	\begin{adjustwidth}{0.0cm}{0.0cm}
		\subfloat[][$t=0.1$]{\includegraphics[width=0.49\linewidth]{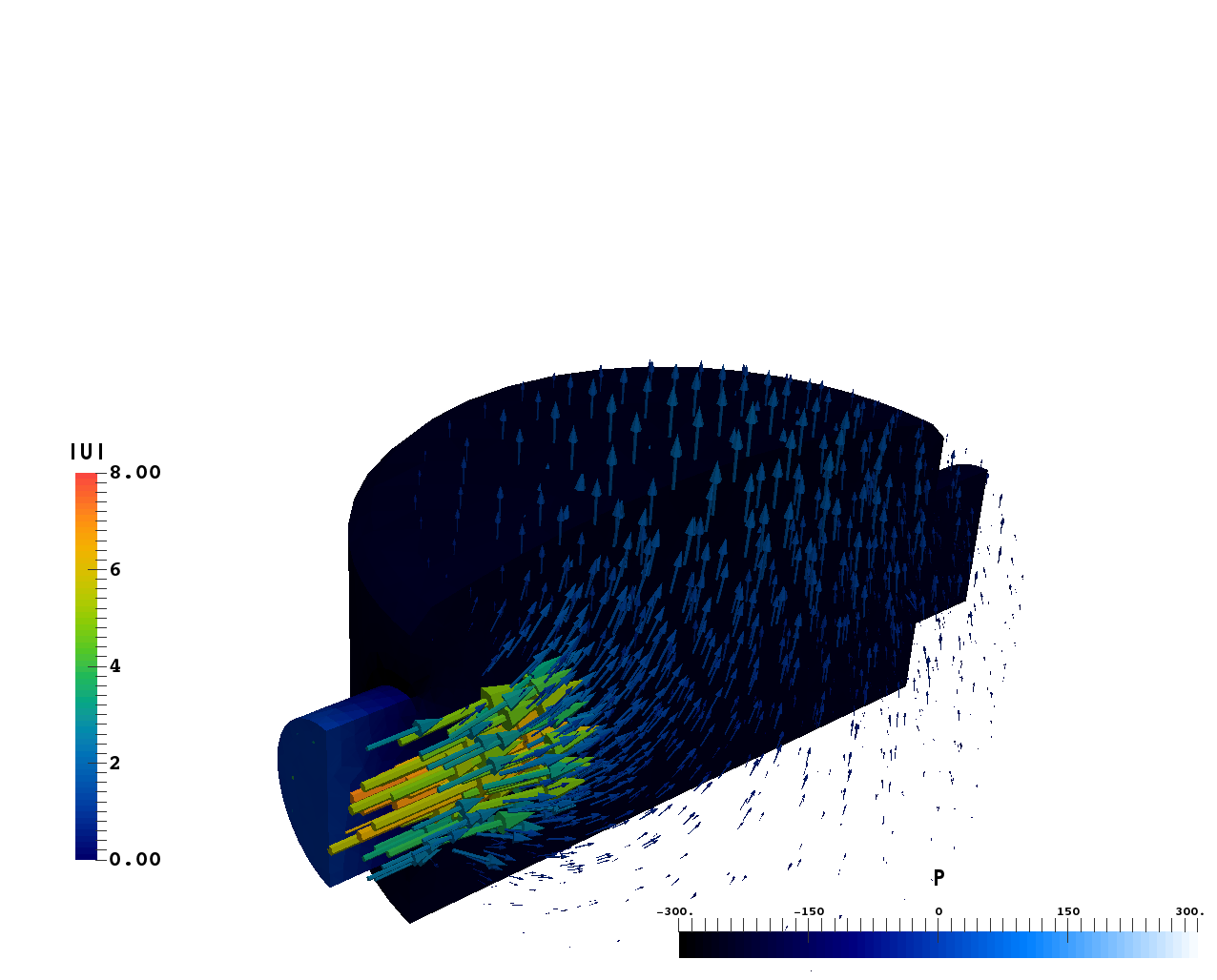}}%
		\subfloat[][$t=0.35$]{\includegraphics[width=0.49\linewidth]{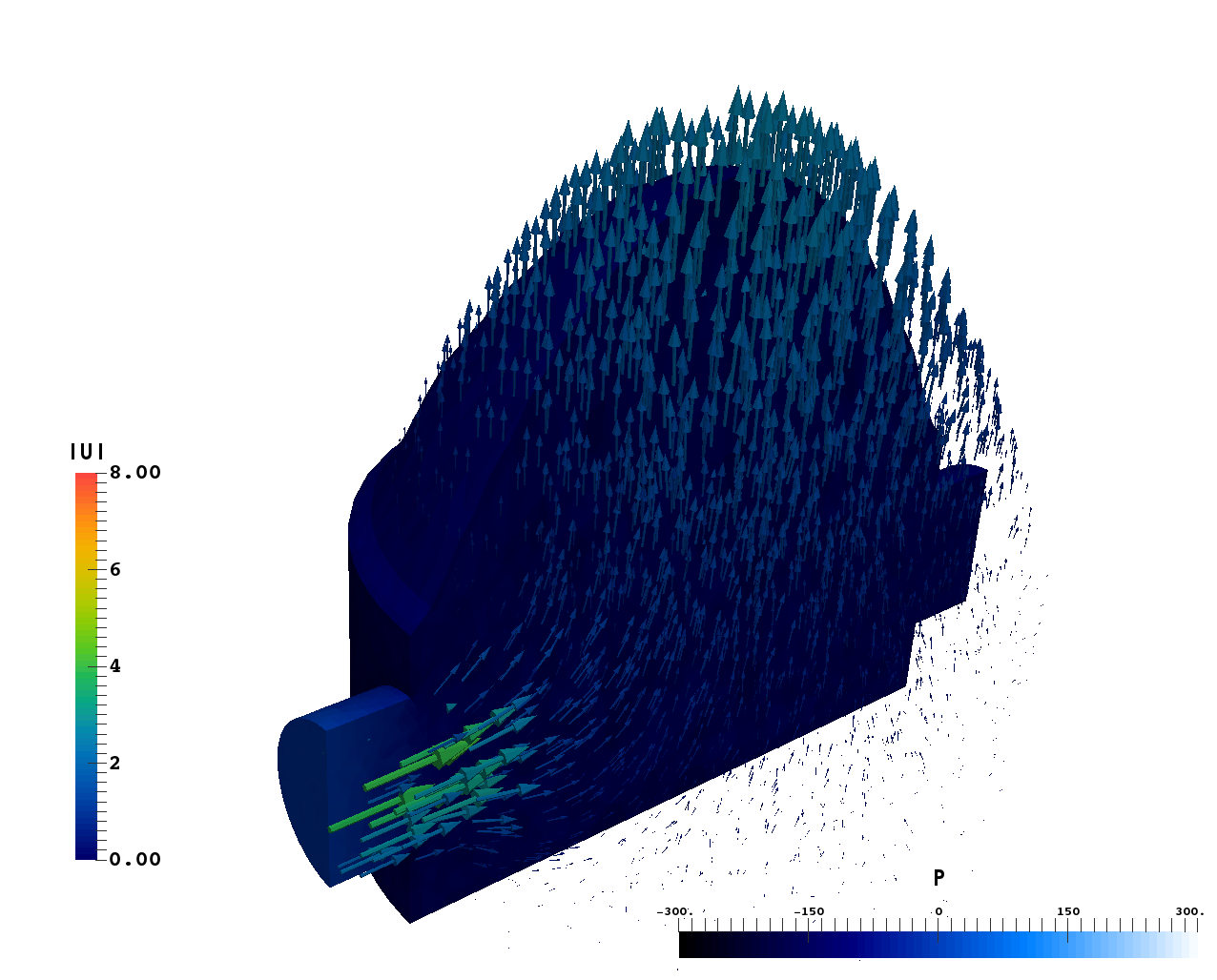}}\\%
		\subfloat[][$t=0.7$]{\includegraphics[width=0.49\linewidth]{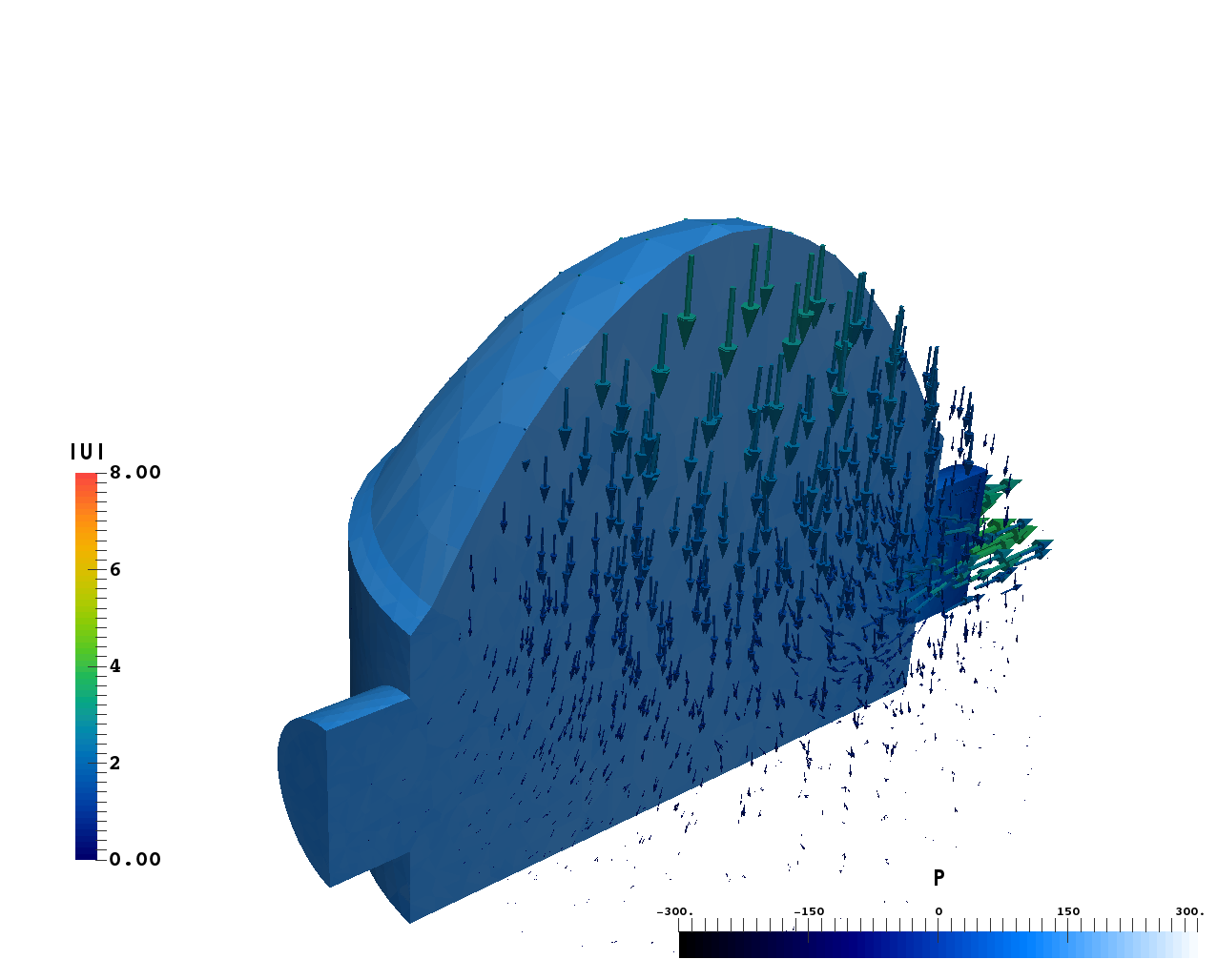}}
		\subfloat[][$t=1.0$]{\includegraphics[width=0.49\linewidth]{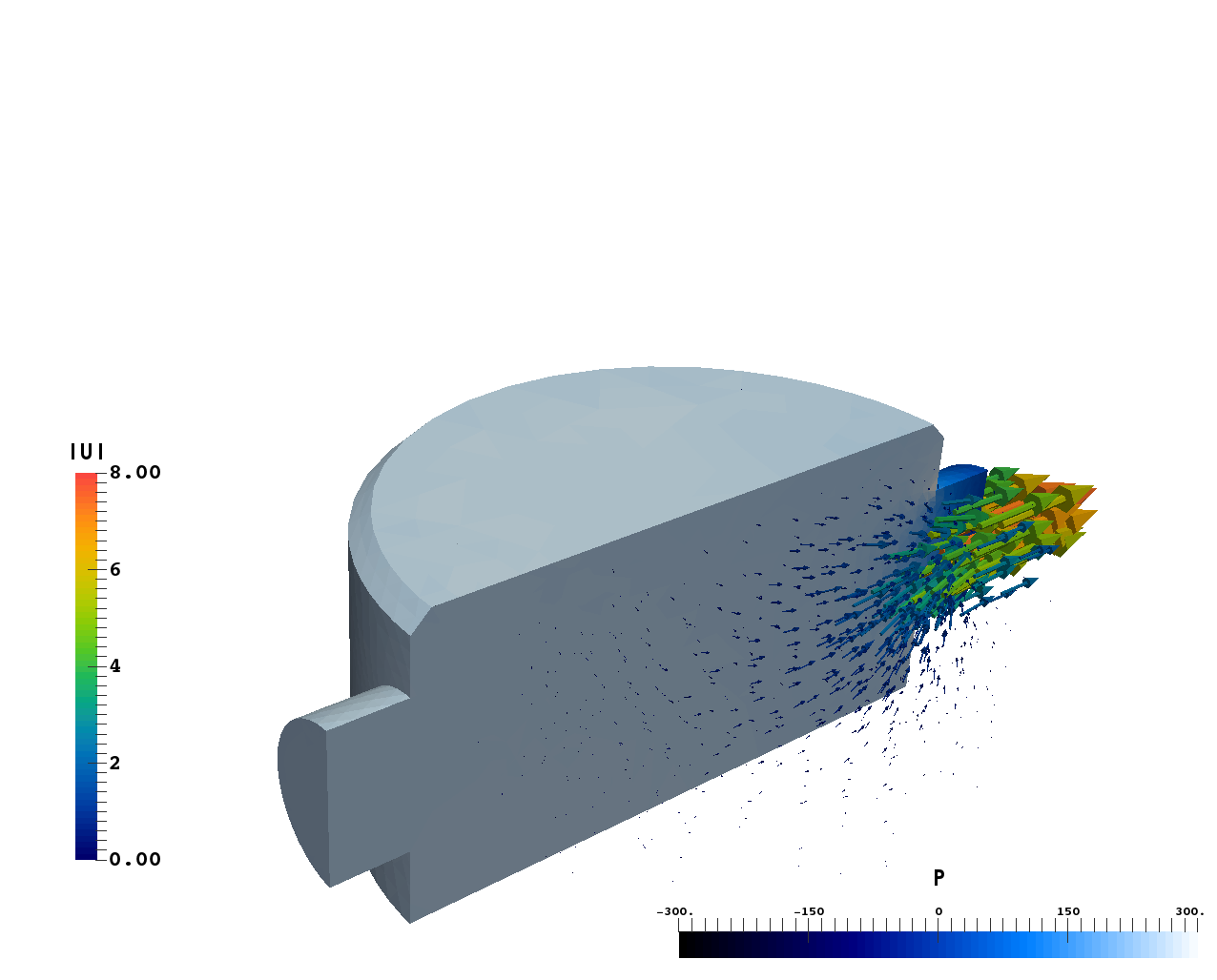}}
	\end{adjustwidth}
	\caption{Snapshots of the solution. Additionally we have cut along the $y$-axis.}
	\label{fig:results:1}
\end{figure*}

\subsection{Second Example}
For the second example we considered a Y-shaped pipe. A schematic view is depicted in Figure \ref{fig:pump_front}. We prescribed the following movement of $\Gamma(t)$:
\begin{align*}
\vec g_\text{mov}(\vec X, t) := \begin{cases}
\vec 0 & \mbox{for } \vec X \notin \Gamma_{D,m} \cup \widetilde{\Gamma_{D,m}} \\
4\frac{\abs{X(2)+3}}{7} \mathrm{sin}(\pi t)^2 \vec e_z & \mbox{for } \vec X \in \Gamma_{D,m} \cup \widetilde{\Gamma_{D,m}}
\end{cases}.
\end{align*}
Some snapshots of the domain movement are depicted in Figure \ref{fig:triangulation_snapshots_ypump}. 
The boundary conditions were set to
\begin{itemize}
	\item $\vec u = \vec 0$ on $\Gamma_D \cup \widetilde{\Gamma_{D,m}}(t)$
	\item $\vec u =  \frac{\partial}{\partial t} \vec g_\text{mov}$ on $\Gamma_{D,m}(t)$
	\item On $\Gamma_{R,\text{in}}$ and $\Gamma_{R,\text{out}}$ we used the Robin boundary conditions discussed in Section \ref{sec:robinBC}
	\item The initial condition for $\vec u$ was set to $\vec u(0,\vec x) = \vec 0$.
\end{itemize}
The resulting 4D mesh consisted of 2618880 pentatopes. With the same ansatz spaces as used for Example 1 we have 39283200 degrees of freedom for $\vec u_h$ and 2618880 degrees of freedom for $p_h$. We needed 107 GMRes-iterations for achieving a relative error of $1E-5$. In Figure \ref{fig:results:2} we have depicted some results.
\begin{figure}[htbp]
	\centering
	\begin{tikzpicture}
	\coordinate (A) at (0,60.0);
	\coordinate (B) at (-11.5192,79.9519);
	\coordinate (C) at (-63.4808,49.9519);
	\coordinate (D) at (-30.0,-8.03848);
	
	\coordinate (E) at (-30.0,-30.0);
	\coordinate (F) at (-30.0,-105.0);
	\coordinate (G) at (30.0,-105.0);
	\coordinate (H) at (30.0,-30.0);
	
	\coordinate (I) at (30.0,-8.03848);
	\coordinate (J) at (63.4808,49.9519);
	\coordinate (K) at (11.5192,79.9519);
	
	\fill[black!10!white] ($0.05*(A)$) -- ($0.05*(B)$) -- ($0.05*(C)$) --  ($0.05*(D)$) -- ($0.05*(E)$) -- ($0.05*(F)$) -- ($0.05*(G)$) -- ($0.05*(H)$) -- ($0.05*(I)$) -- ($0.05*(J)$) --($0.05*(K)$) -- cycle;
	
	\draw[-,color=orange,thick] ($0.05*(A)$) -- ($0.05*(B)$);
	\draw[-,thick] ($0.05*(B)$) -- ($0.05*(C)$);
	\draw[-,color=orange,thick] ($0.05*(C)$) -- ($0.05*(D)$);
	\draw[-,color=orange,thick] ($0.05*(D)$) -- ($0.05*(E)$);
	\draw[-,color=black!30!blue,thick] ($0.05*(E)$) -- ($0.05*(F)$);
	\draw[-,color=black!30!red,thick] ($0.05*(F)$) -- ($0.05*(G)$);
	\draw[-,color=black!30!blue,thick] ($0.05*(G)$) -- ($0.05*(H)$);
	\draw[-,color=orange,thick] ($0.05*(H)$) -- ($0.05*(I)$);
	\draw[-,color=orange,thick] ($0.05*(I)$) -- ($0.05*(J)$);
	\draw[-,thick] ($0.05*(J)$) -- ($0.05*(K)$);
	\draw[-,color=orange,thick] ($0.05*(K)$) -- ($0.05*(A)$);
	\draw[dashed] ($0.05*(E)$) -- ($0.05*(H)$);
	\node at (0,0) {$\Omega(0)$};
	\node[below] at ($0.025*(F)+0.025*(G)$) {$\Gamma_{D,m}(0)$};
	\node[label={[xshift=-0.1cm, yshift=0.1cm]$\Gamma_{R,\mathrm{in}}$}] at ($0.025*(B)+0.025*(C)$) {};
	\node[label={[xshift=0.1cm, yshift=0.1cm]$\Gamma_{R,\mathrm{out}}$}] at ($0.025*(K)+0.025*(J)$) {};
	\node[label={[xshift=-0.4cm, yshift=-0.4cm]$\Gamma_{D}$}] at ($0.025*(C)+0.025*(D)$) {};
	\node[label={[xshift= 0.5cm, yshift=-0.4cm]$\Gamma_{D}$}] at ($0.025*(I)+0.025*(J)$) {};
	
	\node[label={[xshift=-0.8cm]$\widetilde{\Gamma_{D,m}}(0)$}] at ($0.025*(E)+0.025*(F)$) {};
	\node[label={[xshift=0.8cm]$\widetilde{\Gamma_{D,m}}(0)$}] at ($0.025*(G)+0.025*(H)$) {};
		
	\end{tikzpicture}
	\caption{Front view of the initial geometry $\Omega(0)$. Blue boundaries belong to $\Gamma_D$. The dashed line represents the plane $z=-3$. The height from top to bottom is 17. The base of the pipe is located at $z=-10$. The radius of the pipe is 3.}
	\label{fig:ypump_front}
\end{figure}

\begin{figure}[htbp]
	\centering
	\begin{adjustwidth}{4em}{4em}
		\subfloat[][]{\includegraphics[width=0.49\linewidth]{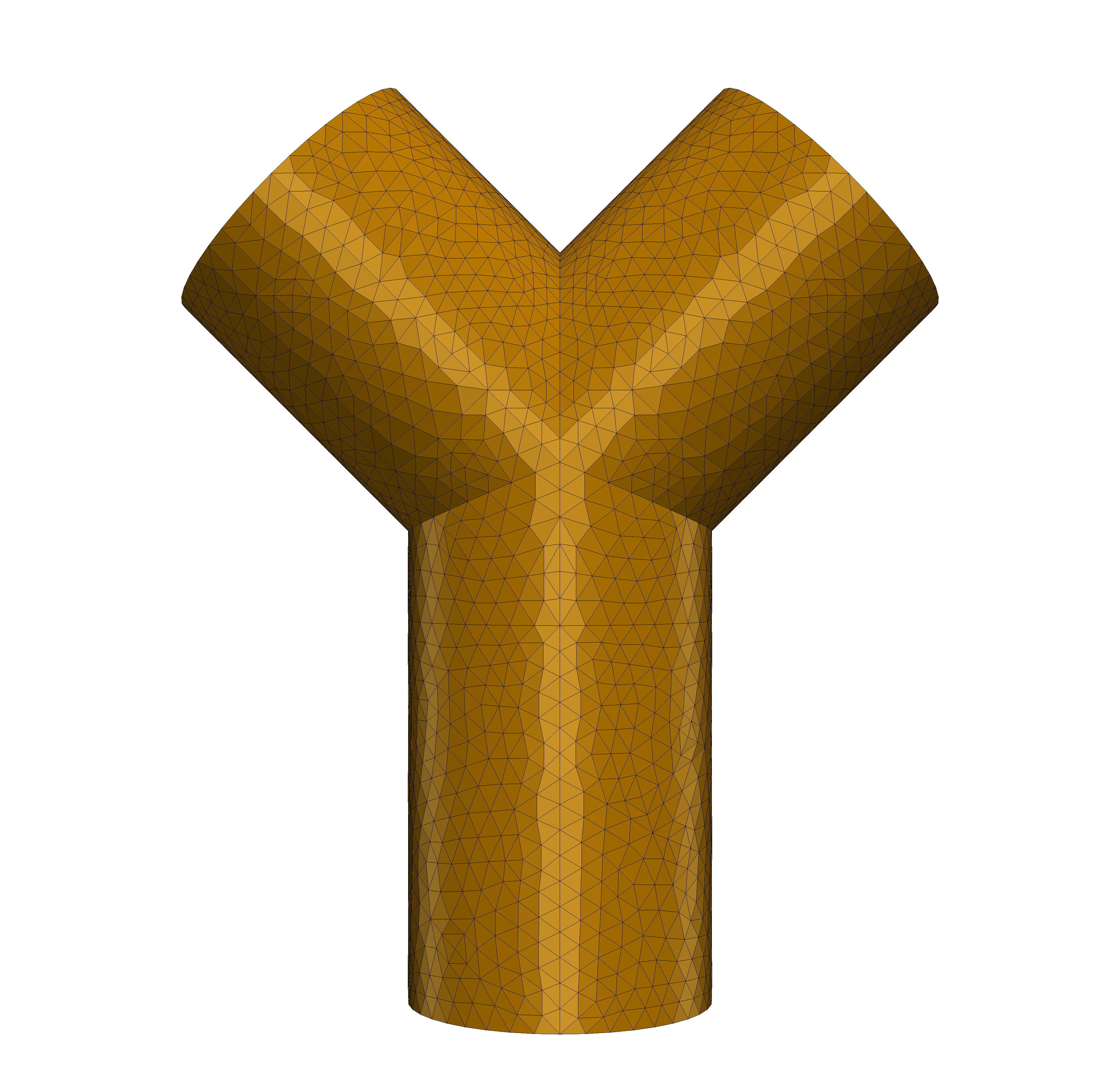}}%
		\subfloat[][]{\includegraphics[width=0.49\linewidth]{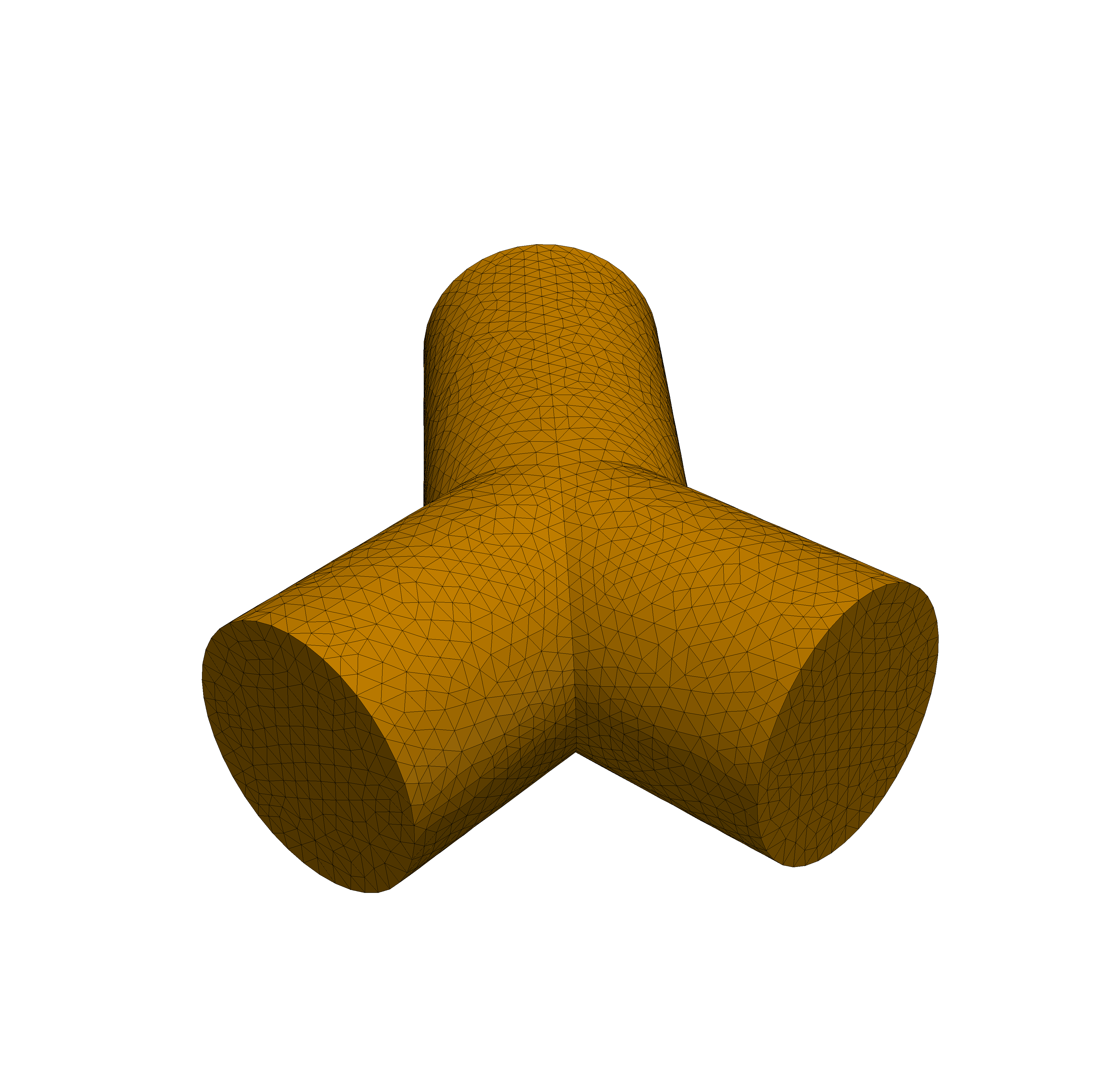}}\\%
		\subfloat[][]{\includegraphics[width=0.49\linewidth]{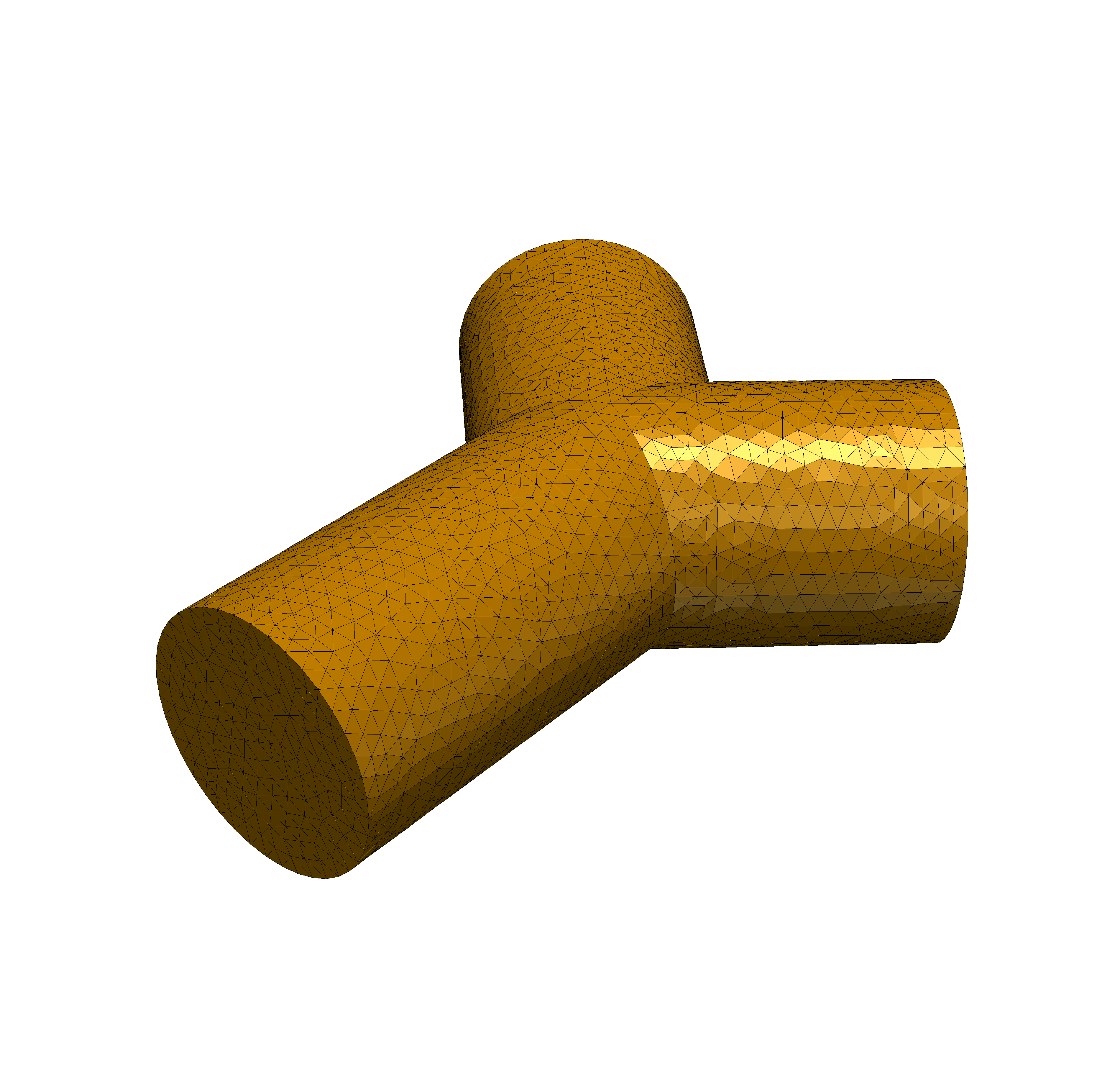}}
		\subfloat[][]{\includegraphics[width=0.49\linewidth]{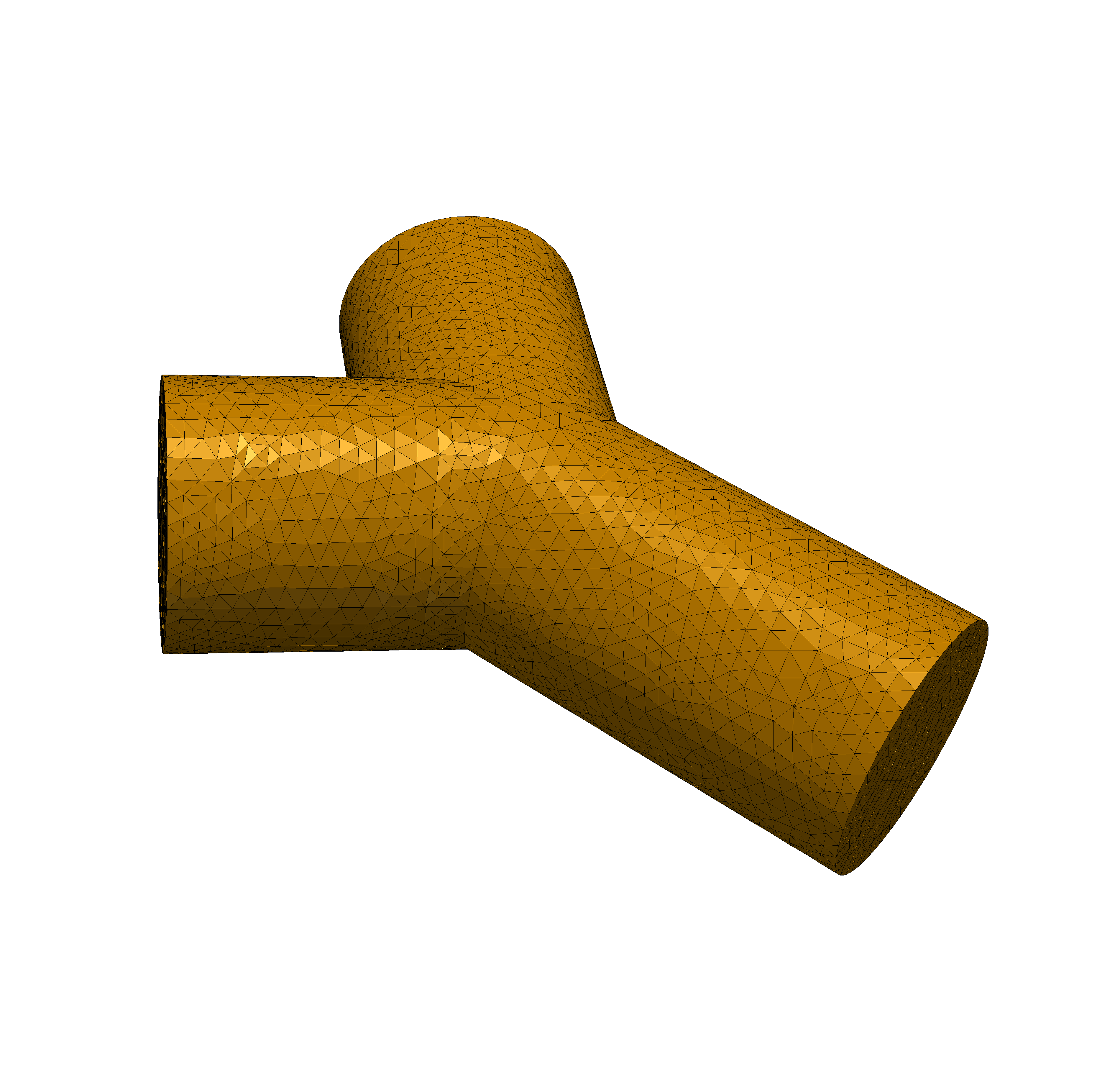}}
	\end{adjustwidth}
	\caption{Initial Triangulation $\Omega(0)$}
	\label{fig:initial_triangulation_ypump}
\end{figure}

\begin{figure}[htbp]
	\centering
	\begin{adjustwidth}{4em}{4em}
		\subfloat[][$\Omega(0.19)$]{\includegraphics[width=0.5\linewidth]{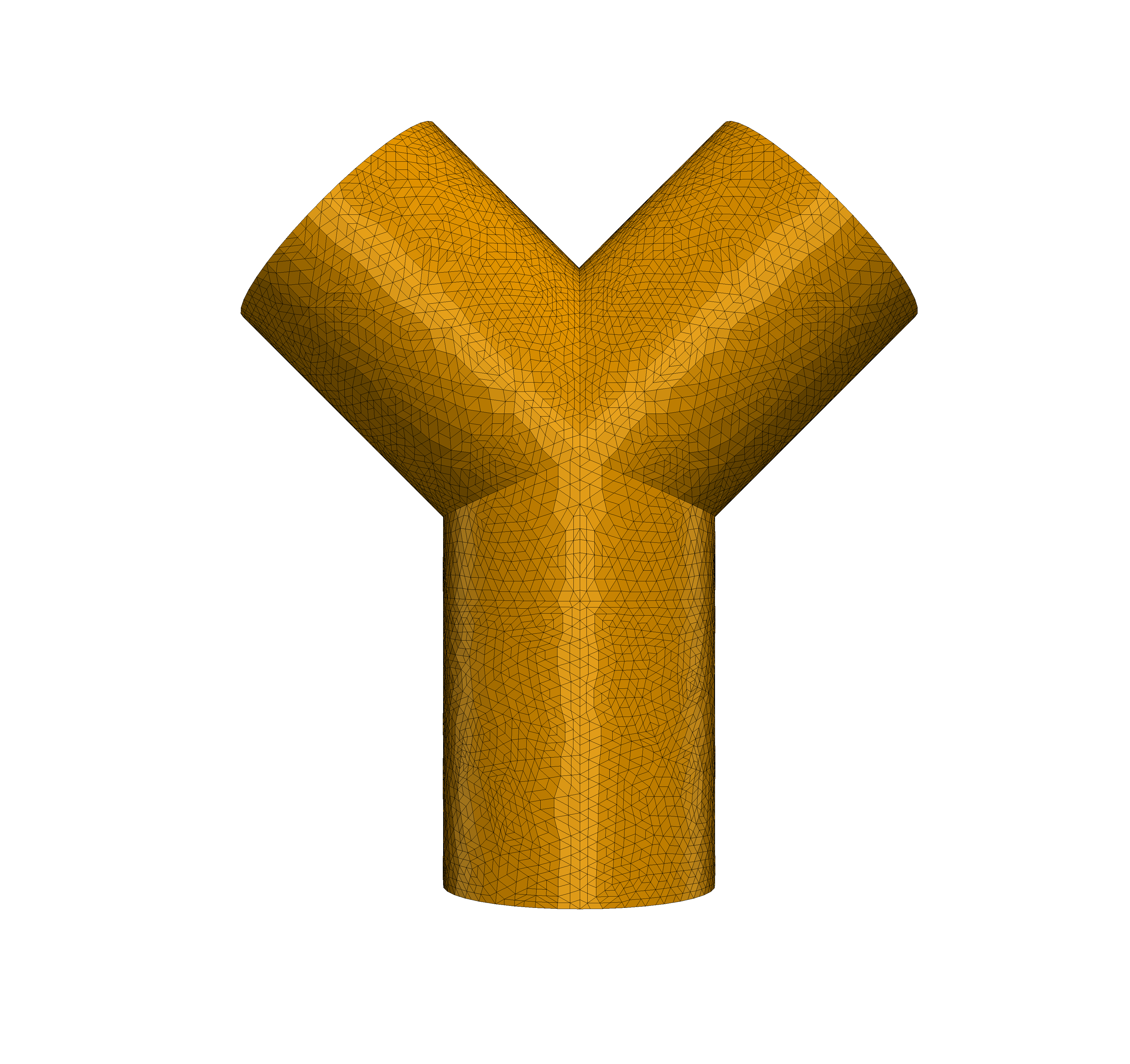}}%
		\subfloat[][$\Omega(0.54)$]{\includegraphics[width=0.5\linewidth]{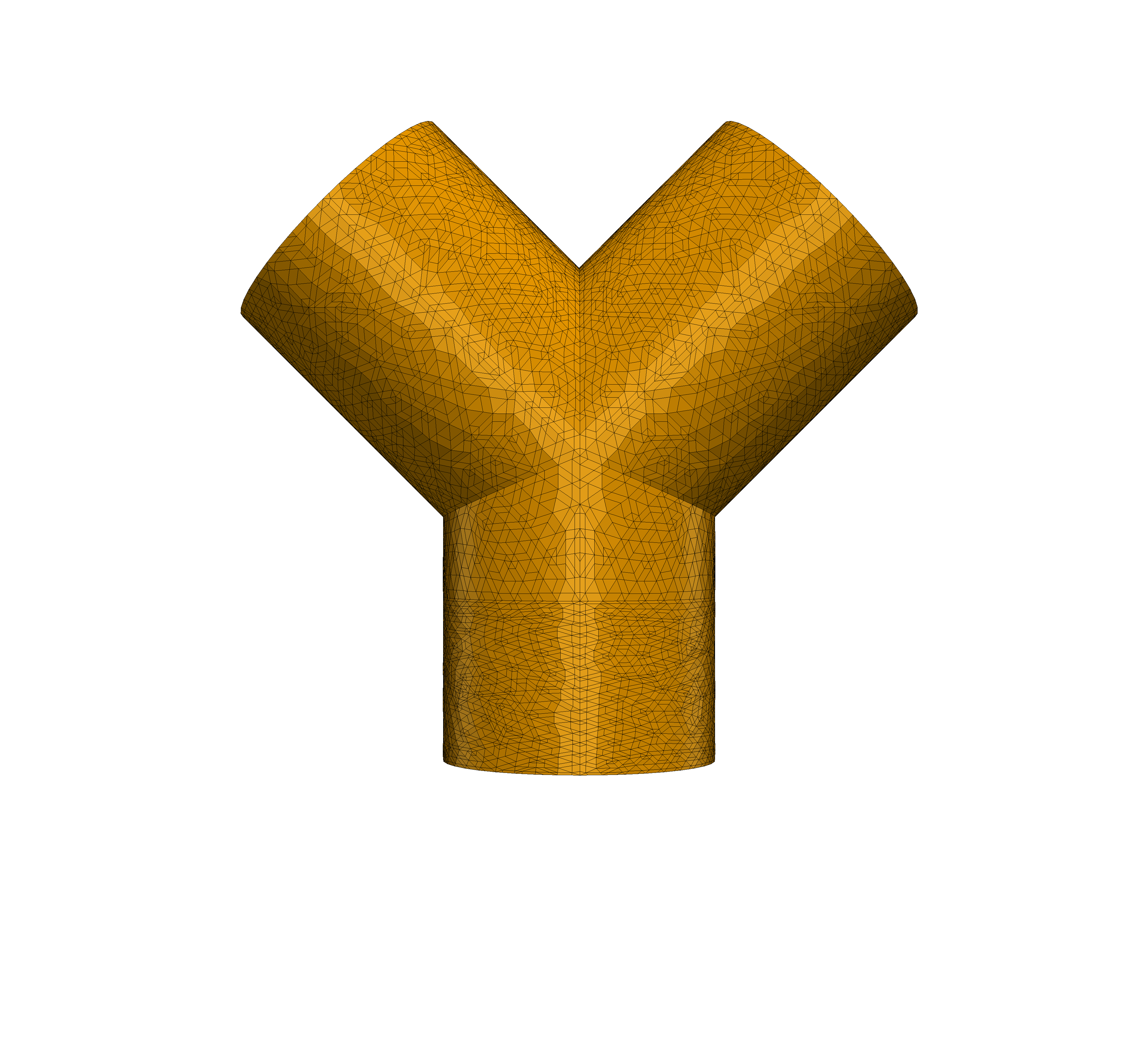}}\\%
		\subfloat[][$\Omega(0.77)$]{\includegraphics[width=0.5\linewidth]{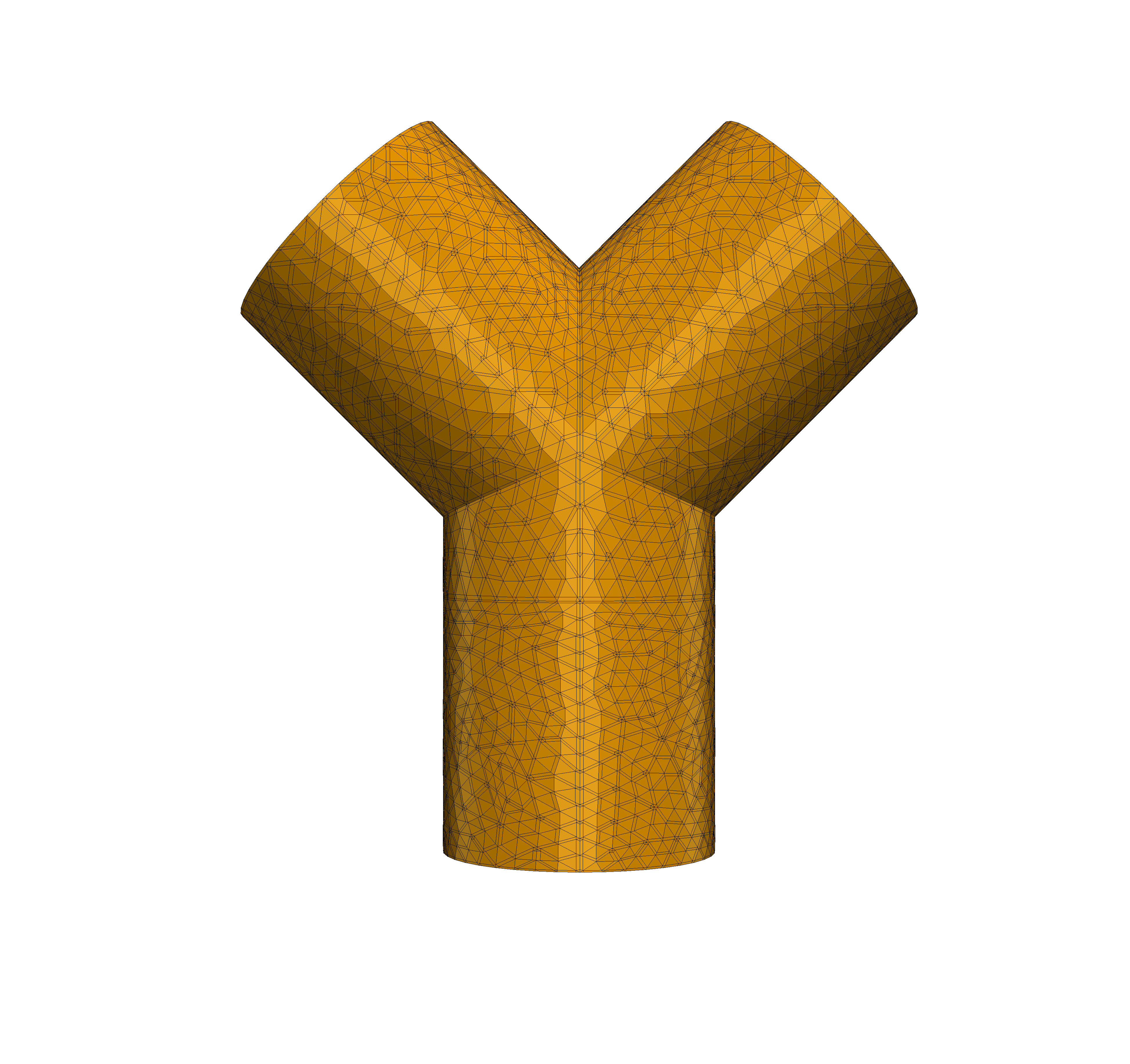}}
		\subfloat[][$\Omega(0.95)$]{\includegraphics[width=0.5\linewidth]{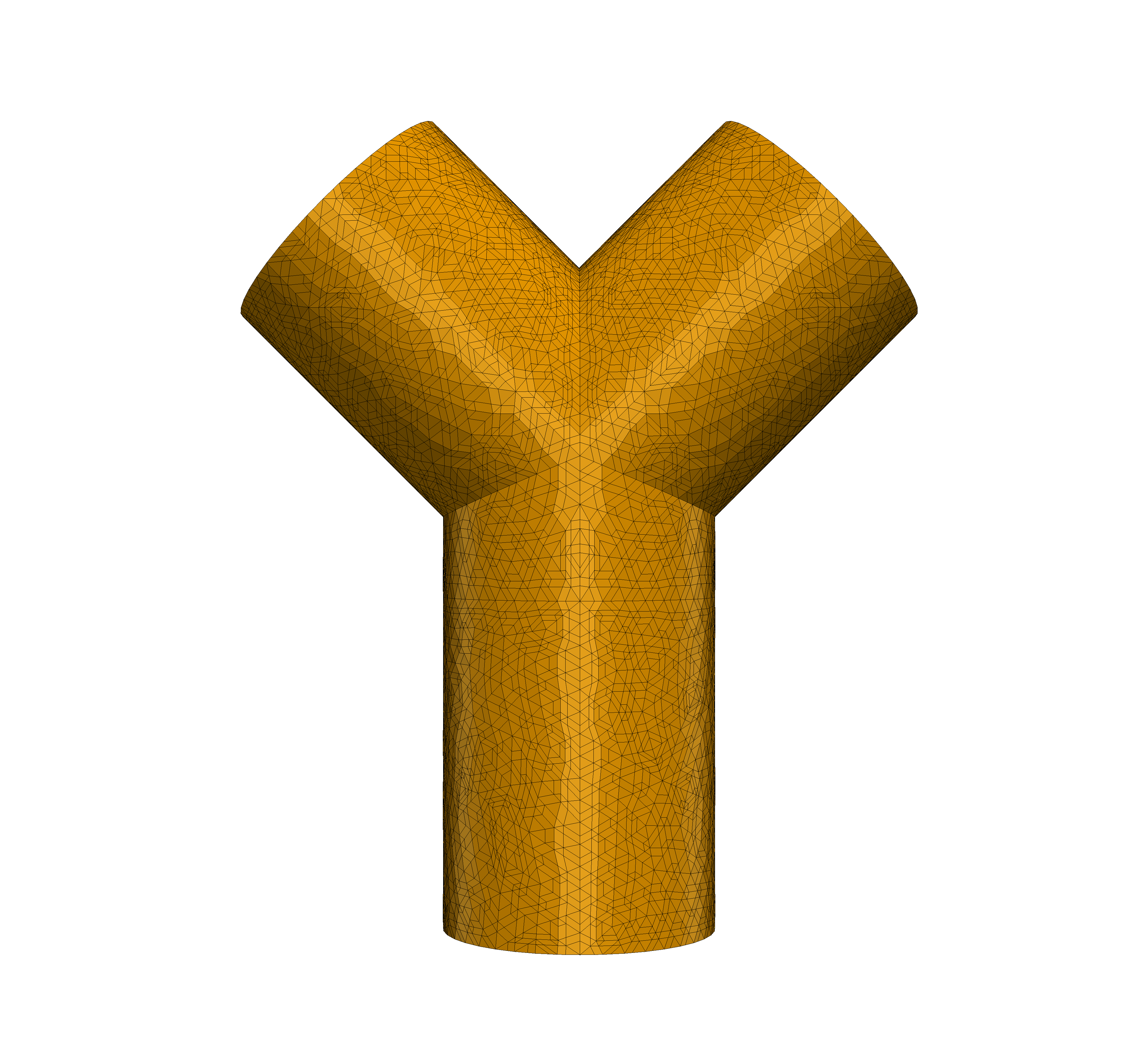}}
	\end{adjustwidth}
	\caption{Snapshots of the Triangulations $\Omega(t)$}
	\label{fig:triangulation_snapshots_ypump}
\end{figure}

\begin{figure*}[htbp]
	\centering
	\begin{adjustwidth}{4em}{4em}
		\subfloat[][$t=0.05$]{\includegraphics[width=0.49\linewidth]{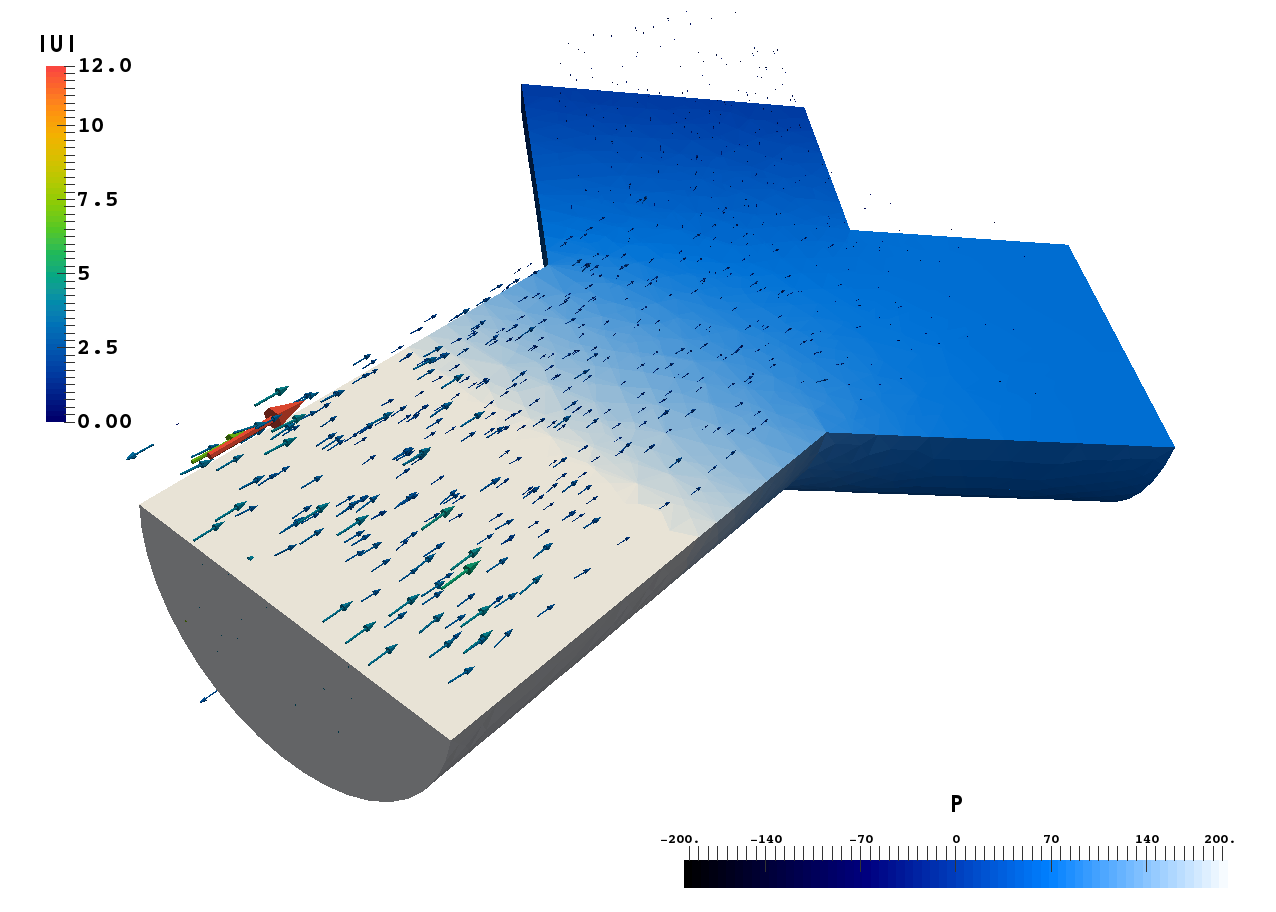}}%
		\subfloat[][$t=0.35$]{\includegraphics[width=0.49\linewidth]{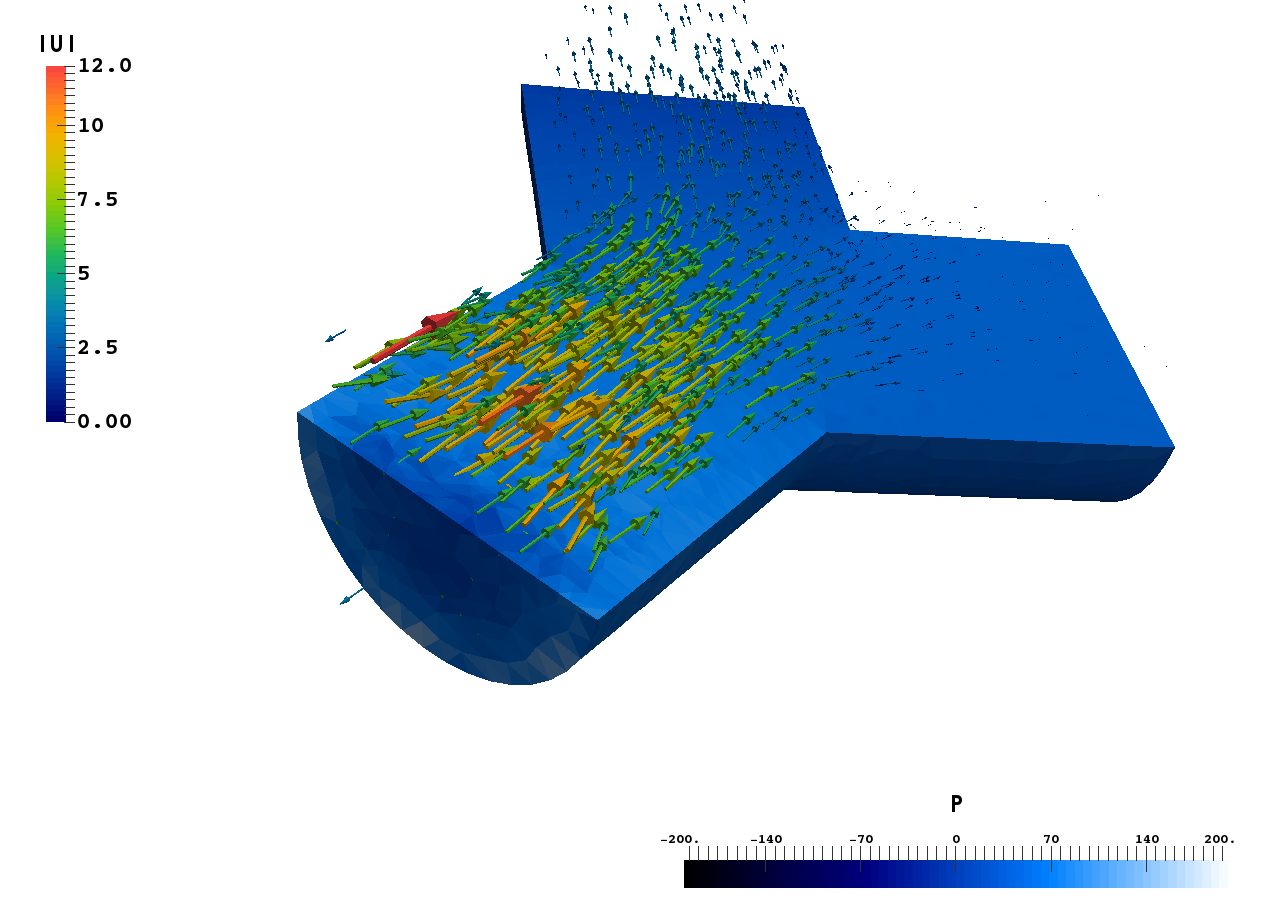}}\\%
		\subfloat[][$t=0.7$]{\includegraphics[width=0.49\linewidth]{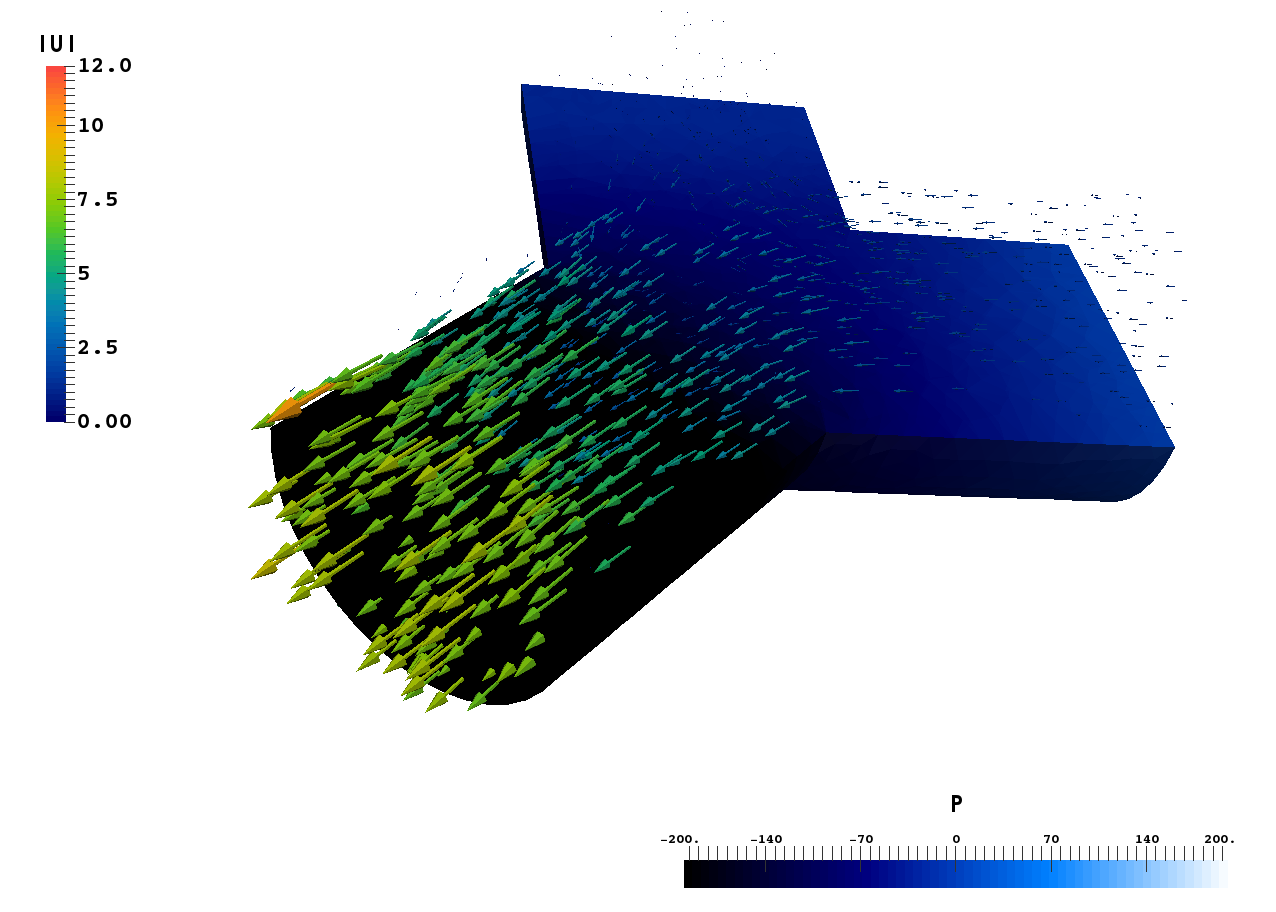}}
		\subfloat[][$t=1.0$]{\includegraphics[width=0.49\linewidth]{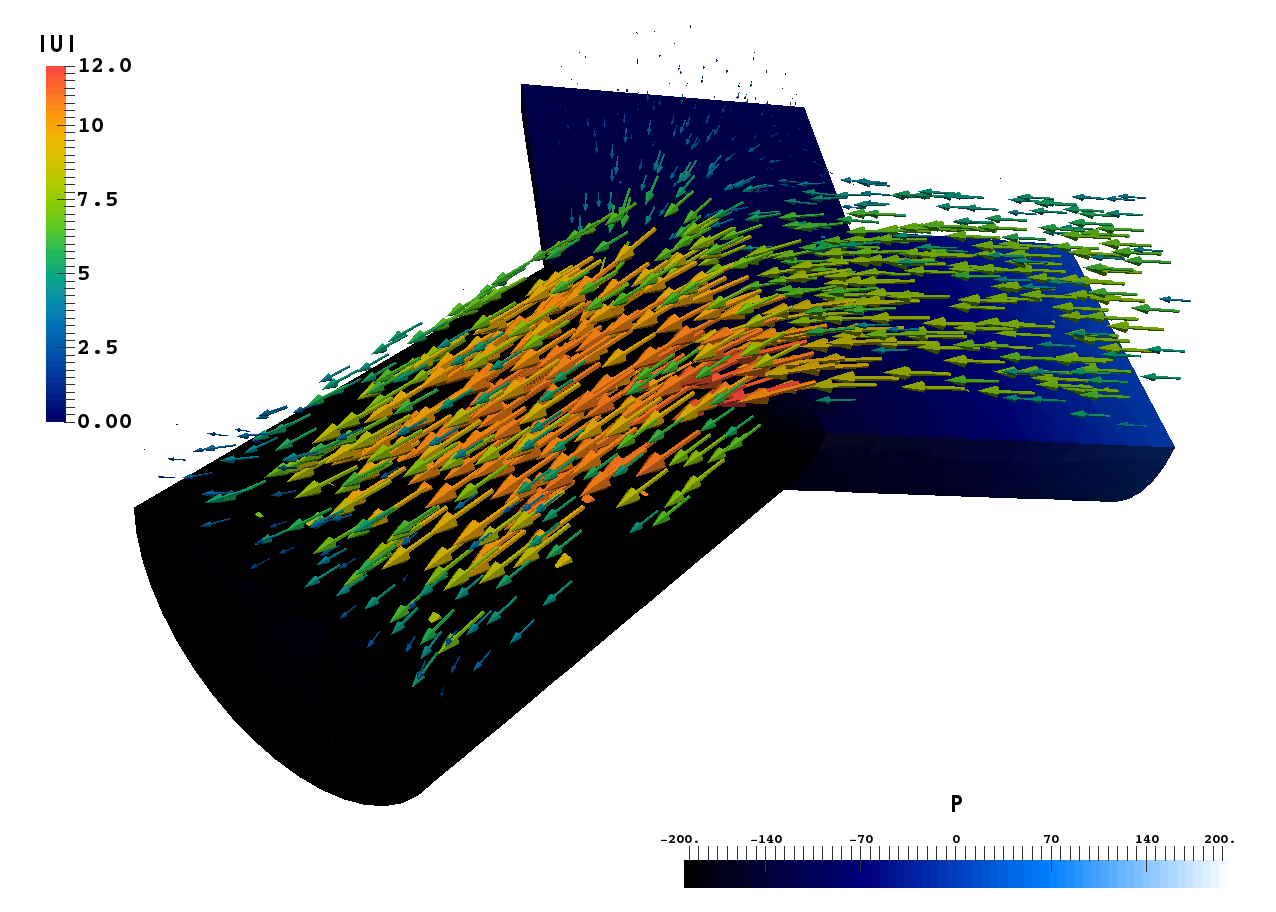}}
	\end{adjustwidth}
	\caption{Snapshots of the solution. Additionally we have cut along the $y$-axis.}
	\label{fig:results:2}
\end{figure*}
\section{Conclusions}

In this paper we have presented a novel approach to construct four-dimensional triangulations for moving domains. This was done by extending the elements of the space triangulation into hyperprisms. Assuming a consistent numbered spatial triangulation we were able to proof that our algorithm produces admissible space-time meshes. We implemented the presented algorithm and applied it to solve the transient Stokes equations with a space-time discontinuous Galerkin finite element method. In the future one could start investigating Navier-Stokes equations. Furthermore, optimal control problems with time dependent partial differential equations render themselves interesting candidates for applying space-time methods, since one has to solve a forward and backward problem which are coupled in space and time. Another attractive aspect of general space-time meshes is the possibility to apply adaptive refinement strategies to resolve local behaviors in space and time. Considering solvers, one could think about domain decomposition approaches or space-time multigrid methods for example, which are a future research topic.

\section*{Acknowledgments}
This research was supported by the grant F3210-N18 from the Austrian Science Fund (FWF).


\bibliographystyle{spmpsci}      
\bibliography{bibliography}   

%
%

\end{document}